\documentclass[10pt,reqno]{amsart}
\usepackage{graphicx}  
\usepackage{enumerate}
\usepackage[colorlinks=true]{hyperref}
\hypersetup{urlcolor=blue, citecolor=red}

\newtheorem{theorem}{Theorem}[section]
\newtheorem{corollary}[theorem]{Corollary}
\newtheorem{lemma}[theorem]{Lemma}
\newtheorem{proposition}[theorem]{Proposition}
\theoremstyle{definition}
\newtheorem{remark}[theorem]{Remark}
\newtheorem{example}{Example}

\usepackage[indentafter]{titlesec}
\titleformat{name=\section}[runin]{}{\thetitle.}{0.5em}{\bfseries}[.]
\titleformat{name=\subsection}[runin]{}{\thetitle.}{0.5em}{\bfseries}[.]

\newcommand{\IN}{\mathrm{in}}

\title[Two limiting essential resources in a self-cycling fermentor]%
{Growth on Two Limiting Essential Resources in a Self-Cycling Fermentor}

\author[T-H. Hsu T. Meadows L. Wang and G.S.K. Wolkowicz]{}

\subjclass[2010]{Primary: 34A37, 34C60, 34D23; Secondary: 92D25.}

\keywords{
Self-cycling fermentation,
impulsive differential equations,  hybrid system, complementary
resources,
state dependent impulses, nutrient driven process,
emptying/refilling fraction,
global attractivity, water purification, wastewater treatment,
optimal yield}

\email{hsut1@math.mcmaster.ca}
\email{tm09ts@gmail.com}
\email{lwang2@unb.ca}
\email{wolkowic@mcmaster.ca}

\thanks{GSKW and LW are supported by  Canadian  Natural Science and Engineering  (NSERC) Discovery Grants.  GSKW also
is supported by an  Accelerator Supplement used to support TH and TM}

\thanks{$^{\ddag}$ Corresponding author: Gail Wolkowicz}

\begin{document}
\maketitle

\centerline{\scshape  Ting-Hao Hsu$^*$, Tyler Meadows$^*$, Lin Wang$^{\dag}$, and Gail S. K. Wolkowicz$^{*\ddag}$}
\medskip
{\footnotesize
\centerline{Department of Mathematics and Statistics$^*$}
\centerline{McMaster University}
\centerline{1280 Main Street West}
\centerline{Hamilton, Ontario}
\centerline{Canada, L8S 4K1}
\medskip
\centerline{Department of Mathematics and Statistics$^{\dag}$}
\centerline{University of New Brunswick}
\centerline{Tilley Hall, 9 MacAulay Lane}
\centerline{PO Box 4400}
\centerline{Fredericton, New Brunswick}
\centerline{Canada, E3B 5A3}
}

\begin{abstract}
A system of impulsive differential equations with state-dependent
impulses is used to model the growth of a single population on
two limiting essential resources in a self-cycling fermentor.
Potential applications include water purification and biological waste remediation.
The self-cycling fermentation process is a semi-batch process
and the model is an example of a hybrid system.  In this case, a
well-stirred tank is partially drained, and subsequently
refilled using fresh medium when the concentration of both
resources (assumed to be pollutants) falls below some acceptable threshold.
We consider the process successful if the threshold for emptying/refilling the reactor
can be reached indefinitely without the
time between  successive emptying/refillings becoming
unbounded and without  interference by the operator. We prove
that whenever the process is successful, the model predicts that
the concentrations of the population and the resources converge
to a positive periodic solution.  We derive conditions for the
successful operation of the process that are shown to be initial
condition dependent and prove that if these conditions are not
satisfied, then the reactor fails. We show numerically that
there is an optimal fraction of the medium drained from the tank
at each impulse that maximizes the output of the process.
\end{abstract}

\section{Introduction}
\label{sec_introduction}
The self-cycling fermentation (SCF) process can be described in two stages: In the first stage, a well-stirred tank is filled with resources and inoculated with microorganisms that consume the resources. When a threshold concentration of one or more indicator quantities is reached, the second stage is initiated. 
The first stage is a batch culture \cite{Cinar2003}.
During the second stage, the tank is partially drained, and subsequently refilled with fresh resources before repeating the first stage. 

SCF is most often applied to wastewater treatment processes, where the
goal is to reduce the concentration of one or more harmful compounds
\cite{Hughes1996,Sarkis1994}. In this application the concentration of
harmful compounds is the most  reasonable threshold quantity,   since acceptable concentrations would typically be given by some government agency. More recently, the SCF process has been used as a means to improve production of some biologically derived compounds \cite{Storms2012,Wang2017}. In these instances, dissolved $\mathrm{O_2}$ content, or dissolved $\mathrm{CO_2}$ content have been used as threshold quantities,  since they are good indicators of when the microorganism approaches the stationary phase in its growth cycle. In both scenarios,
the end goal is to maximize the amount of substrate processed by the
reactor, while maintaining stable operating conditions. SCF has also
been used to culture synchronized microbial cultures
\cite{Sauvageau2010}, where stability of the operating conditions is
much more important than output of the reactor.
It is the first scenario that we model, i.e. we consider the case that the microbial population is used to reduce two harmful compounds to an acceptable level.

Assuming that the time taken to empty and refill the tank is negligible, we can model the SCF process using a system of impulsive differential equations. Smith and Wolkowicz \cite{Smith1995} used this approach to model the growth of a single species with one limiting resource. Fan and Wolkowicz \cite{Fan2007} extended this model to include the possibility that the resource is limiting at large concentrations. C\'ordova-Lepe, Del Valle, and Robledo \cite{Cordova2014} also modeled single species growth in the SCF process, but used impulse dependent impulse times instead of the state dependent impulses used by the other models. 
For references on the theory of impulsive differential equations,
see e.g.\ \cite{Bauinov1989,Bauinov1993,Bauinov1995,Halanay1968,Samoilenko1995}.

When there are two (or more) resources in limited supply, it is
important to think about how the resources interact to promote growth.
If any of the resources can be used interchangeably with the same
outcome, we say the resources are substitutable. For instance, both
glucose and fructose are carbon sources for many bacteria, and can
fulfill the same purpose in bacterial growth. If all of the resources
are required in some way for growth, and the bacteria will die out if
any were missing, we say the resources are essential. For instance, both
carbon and nitrogen are required for growth of many bacteria, but
glucose cannot be used as a nitrogen source, so some other compound such
as nitrate is required. Growth and competition with two essential
resources has been studied in the chemostat \cite{Butler1987}, in the chemostat with delay \cite{Li2000}, and in the unstirred chemostat \cite{Wu2004}. In all of the aforementioned studies, the interaction of essential resources is through Liebig's law of the minimum \cite{Liebig1840}.
To illustrate the law of the minimum, consider a barrel with several staves of unequal length. Growth is limited by the resource in shortest supply in the same way that the capacity of the barrel is limited by length of the shortest stave. 

In this paper we investigate the dynamics of the self-cycling fermentation process
in a semi-batch culture
with two essential resources
that are assumed to be pollutants.
The goal is to reduce both pollutant
concentrations to acceptable levels.
In section \ref{sec_model} we introduce the model. In section \ref{sec_odes} we analyze the system of ordinary differential equations (ODEs)
associated 
with
the model introduced in section \ref{sec_model}. In section \ref{sec_full} we analyze the system of impulsive differential equations introduced in section \ref{sec_model}, and obtain our main results: Theorem \ref{thm_per_exists}, which gives necessary and sufficient conditions for the existence of a unique periodic orbit, and Theorem \ref{Theorem_summary}, which summarizes all of the possible long term dynamics of the model.
In section \ref{sec_max} we demonstrate numerically that the
emptying/refilling fraction can be used to maximize the output of the
SCF process. In section \ref{sec_discussion} we summarize our results
and  discuss the implications of our analysis. All figures were produced using Matlab \cite{MATLAB:2017b}.

\section{Model Formulation}
\label{sec_model}
For a given function $y(t)$ and time $\tau$,
using the standard notation for impulsive equations
we denote by
$\Delta y(\tau)= y(\tau^+)- y(\tau^-)$,
where
\begin{eqnarray*}
  y(\tau^+) \equiv \lim_{t \rightarrow \tau^+} y(t) \  \
  \ \mbox{and} \ \ \   y(\tau^-) \equiv \lim_{t
  \rightarrow \tau^-} y(t).
\end{eqnarray*}

Our model takes the form
\begin{equation}
 \left.\begin{array}{ccl}
  \frac{ds_1(t)}{dt}&=&-\frac{1}{Y_{1}}\min\{f_1(s_1(t)), f_2(s_2(t))\}x(t), \\
  \frac{ds_2(t)}{dt}&=&-\frac{1}{Y_{2}}\min\{f_1(s_1(t)), f_2(s_2(t))\}x(t), \\
  \frac{dx(t)}{dt}&=&(-D+\min\{f_1(s_1(t)), f_2(s_2(t))\})x(t), \nonumber 
\end{array}  \right\} \quad t \ne t_k
\end{equation}
~\vspace{-.25in}
\begin{equation} \label{modeleq}
~\hspace{3in}
\end{equation}
\begin{equation}
\left.\begin{array}{ccl}
  \Delta s_1(t_k)&=&-rs_1(t_k^-)+rs_1^\IN, \\
  \Delta s_2(t_k)&=&-rs_2(t_k^-)+rs_2^\IN,\\
  \Delta x(t_k)&=&-rx(t_k^-), \nonumber
\end{array} \right\} \quad t= t_k
\end{equation}
where $t_k$ are the times at which 
\begin{equation}
\text{either}\quad \big(s_1(t_k)=\bar{s}_1,s_2(t_k)\leq \bar{s}_2\big)
  \quad\text{or}\quad
  \big(s_1(t_k)\leq \bar{s}_1,s_2(t_k)=\bar{s}_2\big).
  \label{cond_impulse}
\end{equation}
Here, $t$ denotes time. The variables   $s_i, \ i=1,2$ denote the concentration
 of the limiting resources (assumed to be pollutants) in the fermentor  as a function of
$t$, with associated parameters $Y_i$,  the   cell yield constants,
 $s^\IN_i$,
the concentrations of  each  limiting resource in the medium added to the tank at the
beginning of each new  cycle,   and $\bar{s}_i$ the threshold
concentrations
  of limiting resource that trigger the emptying and
refilling process. Since we are considering the scenario where both $s_1$ and $s_2$ are pollutants, the emptying and refilling process is only triggered when both concentrations reach the acceptable levels $\bar{s}_1$ and $\bar{s}_2$ set by some environmental protection agency. The variable   $x$ denotes the biomass concentration of the population of
microorganisms that consume the resource at time $t$, assumed to have
death rate $D$.    The emptying/refilling fraction is denoted by $r$. It
is assumed that $D > 0, \ 0<r<1$ \ and for $i=1,2, \ Y_i>0,$
and  
$s^\IN_i>\bar{s}_i > 0$.

We call the times $t_k>0$, impulse times, and when they exist they form an
increasing sequence  that we denote $\{t_k\}_{k=1}^N$.  If
\eqref{cond_impulse} is satisfied at $t=0$ or $s_i(0)<\bar{s}_i, \ i=1,2$, then we assume that there is an
immediate impulse at time $t=0$. 
 We consider  the process to be successful if $N=\infty$ and the time
 between impulses, $t_k-t_{k-1}$ remains bounded. We consider the process  a failure if either there are a finite number of impulses, and hence $N$ is finite, or if the time between impulses becomes unbounded. 

 The two resources are assumed to be limiting essential  resources  (see
e.g.,  Tilman \cite{Tilman1982} or Grover \cite{Grover1997}) also
called complementary resources  (see Leon and Tumpson
\cite{Leontumpson1975}), and as in those studies 
we use Liebig's law of the minimum \cite{Liebig1840}
to model the uptake and growth of the microbial
population.

We assume that each response function $f_j(s)$, $j=1, 2$, in
(\ref{modeleq}) satisfies:
\begin{enumerate}
    \item[(i)] $f_j:\mathbb R_+\to \mathbb R_+$ is continuously differentiable;
     \item[(ii)] $f_j(0)=0$ and $f_j'(s)>0$ for $s>0$;
\end{enumerate}
 Define $\lambda_i$, $i=1,2$,
to be the value of each resource that satisfies $f_i(\lambda_i)=D$,
and refer to each $\lambda_i$ as a ``break-even concentration".
If $f_j$ is bounded below $D$, then we define the corresponding $\lambda_j=\infty$.

Between two consecutive impulses the system is governed by a system of
ordinary differential equations (ODE) that models a batch fermentor \cite{Cinar2003}, 
 
    \begin{eqnarray}
    \frac{ds_1(t)}{dt}&=&-\frac{1}{Y_{1}}\min\{f_1(s_1(t)), f_2(s_2(t))\}x(t), \nonumber\\
    \frac{ds_2(t)}{dt}&=&-\frac{1}{Y_{2}}\min\{f_1(s_1(t)),
    f_2(s_2(t))\}x(t), \label{odes}\\
    \frac{dx(t)}{dt}&=&(-D+\min\{f_1(s_1(t)), f_2(s_2(t))\})x(t). \nonumber 
    \end{eqnarray}
We will refer to system \eqref{odes} as the associated ODE system.

\section{Dynamics of System \eqref{odes}}
\label{sec_odes}
First we show that system \eqref{odes} is well-posed.

\begin{proposition}
Given any positive initial conditions $(s_1(0),s_2(0),x(0))$,
the solution $(s_1(t),s_2(t),x(t))$ of \eqref{odes} is defined for all $t\ge 0$
and remains positive.

Furthermore, $\lim_{t\to \infty}(s_1(t),s_2(t),x(t))$ exists, is initial
condition dependent,\\ $\lim_{t\to \infty} x(t) = 0$, and 
$\lim_{t\to \infty} s_i(t)<\lambda_i$ for at least one
$i\in\{1,2\}$.
\label{prop_positivity}
\end{proposition}

\begin{proof}
Since the vector field in \eqref{odes} is locally Lipschitz, the positivity of $(s_1(t),s_2(t),x(t))$
follows from  the  standard theory for the existence and uniqueness of solutions of  ODEs
(see e.g., \ \cite{Perko1996}).
Also observe that \[
  \frac{d}{dt}\left(
    x(t)+ \frac{Y_1}{2}s_1(t)+ \frac{Y_2}{2}s_2(t)
  \right)
  = -Dx(t)< 0,
\] so the solution $(s_1(t),s_2(t),x(t))$ exists and is bounded for $t$ in $[0,\infty)$.

From \eqref{odes},  $s_i'(t)< 0$, $i=1,2$. 
By the positivity of $s_i(t)$,  $\lim_{t\to \infty} s_i(t)\ge0$ exists.
Denote $\lim_{t\to \infty} s_i(t)=s_i^*$, $i=1,2$.

We claim that $\min\{f_1(s_1^*),f_2(s_2^*)\}< D$.
Suppose not. Then, $\min\{f_1(s_1^*),f_2(s_2^*)\}\ge D$.
Since both $s_1(t)$ and $s_2(t)$ are strictly decreasing functions,
it follows that $\min\{f_1(s_1(t)),f_2(s_2(t))\}> D$ for all $t\ge 0$.
From the equation of $x'(t)$ in \eqref{odes}, $x(t)$
 would be a strictly increasing function.
Then,   $s_1'(t)<\frac{-D}{Y_1}x(0)$
for all   $t\ge 0$,
contradicting the positivity of $s_1(t)$. Hence, $s_i^*<\lambda_i$ for
at least one $i\in\{1,2\}$.

Since $\min\{f_1(s_1^*),f_2(s_2^*)\}< D$, define
$\gamma=-D+\min\{f_1(s_1^*),f_2(s_2^*)\}<0$.
By the equation for $x'(t)$ in \eqref{odes},
$x'(t)<\frac{\gamma}{2} x(t)<0$, for all sufficiently large $t$.
Hence $x(t)\to 0$ as $t\to \infty$.
\end{proof}

Note that from \eqref{odes}, $Y_1\frac{ds_1}{dt}= Y_2\frac{ds_2}{dt}$. 
Define \begin{equation}
  R_{12}=\frac{Y_{2}}{Y_{1}}
  \quad\text{and}\quad
  R_{21}=\frac{Y_{1}}{Y_{2}}=\frac{1}{R_{12}}.
  \label{def_R12}
\end{equation}
Then, every trajectory of \eqref{odes} satisfies $\frac{ds_2}{ds_1}=R_{21}$.
\begin{lemma}
\label{lem_odes}
Let $(s_1(t),s_2(t),x(t))$ be a  solution of \eqref{odes}
on an interval $t\in [t_0,t_1]$ with positive
initial conditions.
Then, 
\begin{equation}
s_1(t)=s_1({t_0})+R_{12}(s_2(t)-s_2({t_0})),
  \label{s1s2relation}
\end{equation}
or equivalently \begin{equation}
  s_2(t)=s_2({t_0})+R_{21}(s_1(t)-s_1({t_0})),
  \label{s2s1relation}
\end{equation}
\begin{equation}
  x(t_1)-x(t_0)
= Y_1\int_{s_1(t_1)}^{s_1(t_0)}
\left(1-\frac{D}{ \min\big\{ f_1(v),f_2\big(s_2(t_1)+
R_{21}(v-s_1(t_1))\big) \big\} }\right)dv,
  \label{x_diff_s1}
\end{equation}
or equivalently \begin{equation}
  x(t_1)-x(t_0)
  = Y_2\int_{s_2(t_1)}^{s_2(t_0)}\left(1-\frac{D}{\min\{f_1\big(s_1(t_1)+
R_{12}(v-s_2)\big),f_2(v)\}}\right)dv.
  \label{x_diff_s2}
\end{equation}
\end{lemma}

\begin{proof}
Solving the separable ODE,  $\frac{d s_1}{d s_2} = R_{12}$,  yields
\eqref{s1s2relation}, and solving $\frac{d s_2}{d s_1} = R_{21}$,
yields \eqref{s2s1relation}.  Dividing the $s_1'(t)$ equation in
\eqref{odes} by the  $x'(t)$
equation, substituting for $s_2(t)$ using  
\eqref{s2s1relation},
and then integrating both sides,
yields \eqref{x_diff_s1}.  We obtain \eqref{x_diff_s2}
similarly. 
\end{proof}

\section{Analysis of the Full System \eqref{modeleq} }
\label{sec_full}

\begin{figure}[htbp]
\centering
\includegraphics[width=.75\textwidth]{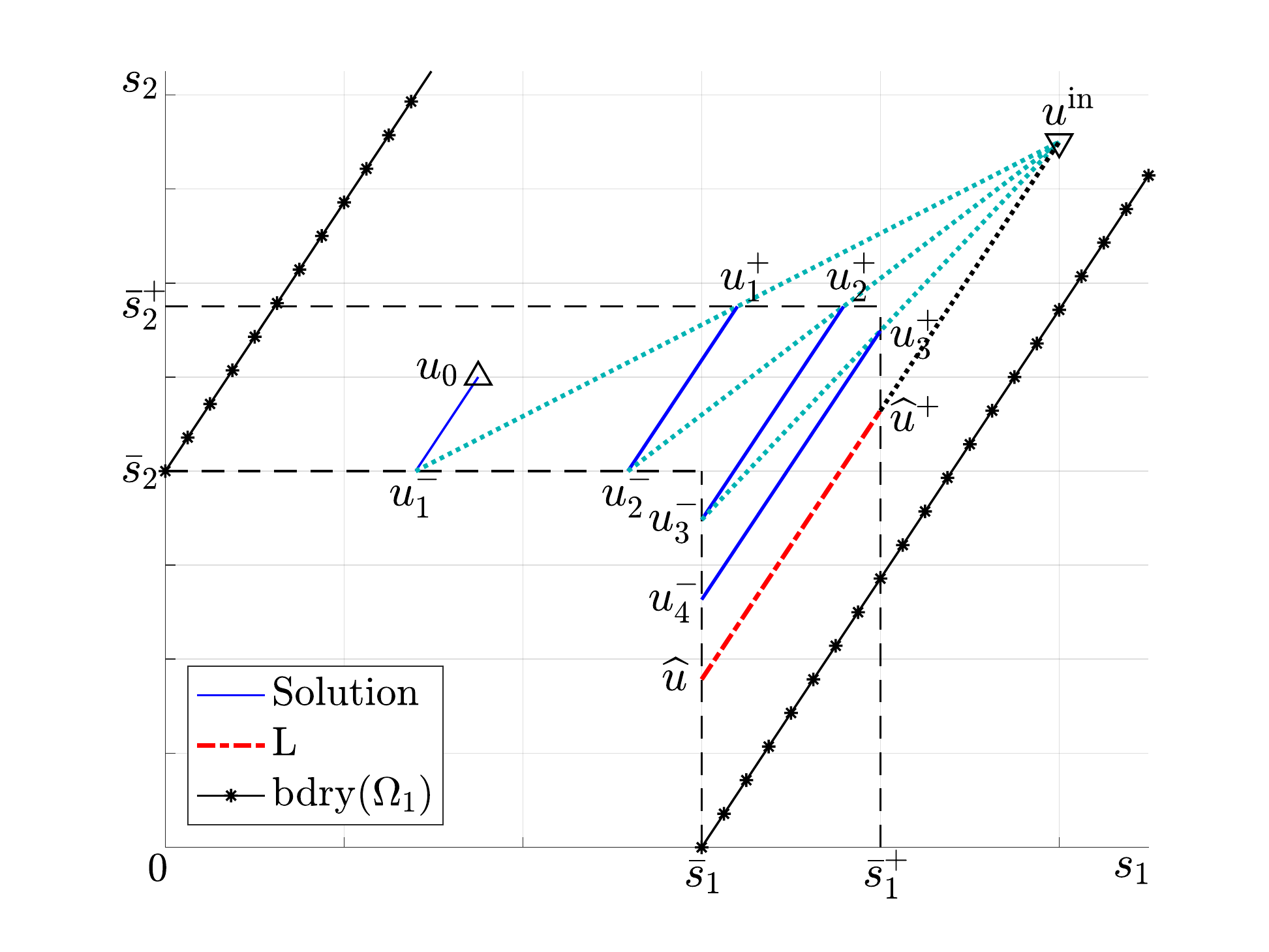}
\caption{%
$u_0=(s_1^0,s_2^0)$,
indicated by $\triangle$,
is the initial condition.
$u^\IN=(s_1^\IN,s_2^\IN)$,
indicated by $\bigtriangledown$,
is the input concentration.
$u_k^{\pm}=(s_1(t_k^{\pm}),s_2(t_k^{\pm})), \ k=1,2,\dots,$
satisfy $u_{k}^+= (1-r)u_k^-+ ru^{\mathrm{in}}$,
where $u^{\mathrm{in}}= (s_1^{\mathrm{in}},s_2^{\mathrm{in}})$.
Each connected piece of the solution
has slope $R_{21}$.
$\widehat{u}=(\bar{s}_1,\widehat{s_2})=(\bar{s}_1,s_2^\IN-R_{21}(s_1^\IN-\bar{s}_1)$
and $\widehat{u}^+=(\bar{s}_1^+,\widehat{s}_2^+)=((1-r)\bar{s}_1+r
s^\IN_1,s_2^\IN-(1-r)R_{21}(s^\IN_1-\bar{s}_1))$.
The set $\Omega_1$ lies above and to the right of
$\Gamma^-$ and between the two lines  with slope
$R_{21}$, through $(0,\bar{s}_2)$ and
$(\bar{s}_1,0)$,  respectively.
}
\label{fig_lines}
\end{figure}

First we visualize solutions of \eqref{modeleq} in the $s_1$-$s_2$ plane
as  illustrated  in Figure~\ref{fig_lines}.
Given any solution $(s_1(t),s_2(t),x(t))$ of \eqref{modeleq}
with  positive initial conditions,
let $t_0= 0$
and let $t_k$, $k\in \mathbb N$, denote the $k$th impulse time
if it exists.

Let $u_k^\pm= (s_1(t_k^\pm),s_2(t_k^\pm))$.
By \eqref{s2s1relation},
the trajectory of $(s_1(t),s_2(t))$, $t\in (t_k,t_{k+1})$,
is a line segment with slope $R_{21}$
and endpoints $u_k^+$ and $u_{k+1}^-$.
The conditions for impulses to occur are given in   \eqref{cond_impulse}.
Therefore, each point $u_k^-$ lies in
the following union of the two horizontal and vertical line segments:
\begin{equation*}
  \Gamma^-\equiv
  \big\{(s_1,\bar{s}_2): s_1\in [0,\bar{s}_1]\big\}
  \cup \big\{(\bar{s}_1,s_2): s_2\in [0,\bar{s}_2]\big\}.
\end{equation*}
Define
$$ u^\IN= (s_1^\IN,s_2^\IN) \quad \mbox{and} \quad 
\bar{s}_i^+= (1-r)\bar{s}_i+ rs_i^\IN, \ i=1,2.
$$
By the definition of  $\Delta s_i$ given in \eqref{modeleq},
\begin{equation}
  u_k^+= (1-r) u_k^-+ ru^\IN.
  \label{eq_ukplus}
\end{equation} 
This implies that each $u_k^+$ lies in
the following union of horizontal and vertical line segments:
\begin{equation*}
  \Gamma^+\equiv
  \big\{(s_1,\bar{s}_2^+): s_1\in [0,\bar{s}_1^+]\big\}
  \cup \big\{(\bar{s}_1^+,s_2): s_2\in [0,\bar{s}_2^+]\big\}.
\end{equation*}
Therefore, if impulses occur indefinitely, then
the total trajectory of $(s_1(t),s_2(t))$, $t\in [t_1,\infty)$,
is a countable union of line segments 
with slope $R_{21}$
and endpoints in $\Gamma^-\cup \Gamma^+$, (i.e., $u_k^+ \in \Gamma^+$ and
$u_k^-\in \Gamma^-$.) 

For any positive solution $(s_1(t),s_2(t),x(t))$ of \eqref{modeleq}
with $(s_1(0),s_2(0))$ lying between the coordinate axes and $\Gamma^-$,
i.e., $0\le s_1(0)\le \bar{s}_1$ and $0\le s_2(0)\le \bar{s}_2$,
an impulse occurs immediately at $t=0$,
and so, after at most a finite number of impulses, $(s_1(0^+),s_2(0^+))$ lies above or to the
right of $\Gamma^-$.
In the rest of this section,
we therefore assume $(s_1(0),s_2(0))$ lies above or to
the right of $\Gamma^-$, i.e., $s_i(0)>\bar{s}_i,$ for at least one $i\in\{1,2\}$.

The following proposition asserts that
system \eqref{modeleq} does not exhibit the phenomenon of {\it beating}.
That is,
the system possesses no solution
with impulse times
that form an increasing sequence with a finite accumulation point.

\begin{proposition}
Assume that $(s_1(t),s_2(t),x(t))$ is a positive solution of
\eqref{modeleq} with an infinite number of impulse times, $\{t_k\}_{k=1}^\infty$.
Then $\lim_{k\to\infty}t_k= \infty$.

\end{proposition}
\begin{proof}
Between impulses, $s_1$ and $s_2$ are strictly decreasing for all $x(t)>0$, and therefore we can solve the first equation in \eqref{odes} for the time between impulses:
\begin{equation}\label{eq:timeseparation}
t_{k+1} - t_k = Y_1\int_{s_1(t_{k+1})}^{s_1(t_k)}
    \frac{1}{\min\big\{f_1(v),f_2(s_2(t_{k+1})+R_{21}(v-s_1(t_{k+1}))\big\}\; X_k(v)}
  dv
\end{equation}
where $X_k(v)$ is
defined by $X_k(s_1(t))=x(t)$ for $t\in (t_k,t_{k+1})$.

To show that $\{t_k\}_{k=1}^\infty$
has no finite accumulation point,
it suffices to show that there exists positive constants, $M_1$, $M_2$ and $m$,
such that
\begin{equation*}
  s_1(t_k^+)-s_1(t_{k+1}^-)> m
  \quad\text{for}\ k=1,2,\dots
\end{equation*}
and
\begin{equation*}
  \min\{f_1(s_1(t)),f_2(s_2(t))\}< M_1,
  \quad
  x(t)< M_2,
  \quad\text{for}\ t\ge 0,
\end{equation*}
since then the difference $t_{k+1}-t_k$ is greater than $\frac{Y_1m}{M_1M_2}$.

Since
$s_1(t_k^-)\le \bar{s}_1$, 
$s_2(t_k^-)\le \bar{s}_2$,
and either $s_1(t_k^+)= \bar{s}_1^+$ or $s_2(t_k^+)=\bar{s}_2^+$,
we have
\begin{equation*}
  \text{either}\quad
  s_1(t_k^+)- s_1(t_{k+1}^-)\ge \bar{s}_1^+-\bar{s}_1
  \quad\text{or}\quad
  s_2(t_k^+)- s_2(t_{k+1}^-)\ge \bar{s}_2^+-\bar{s}_2.
\end{equation*}
From \eqref{s1s2relation} it follows that
\begin{equation*}
  s_1(t_k^+)- s_1(t_{k+1}^-)\ge \min\{\bar{s}_1^+-\bar{s}_1,R_{12}(\bar{s}_2^+-\bar{s}_2)\}
  \equiv m.
\end{equation*}

Since after the first impulse occurs,
$(s_1(t),s_2(t))$ is bounded by $\Gamma^+$,
the existence of $M_1$ follows from the continuity of $f_1$ and $f_2$.
By \eqref{x_diff_s1},
there exists $M_0>0$ such that
\begin{equation}
  x(t)<x(t_k^+)+M_0,
  \quad \text{for}\ t\in (t_k,t_{k+1}).
  \label{ineq_xM0}
\end{equation}
From the relation that $x(t_k^+)=rx(t_k^-)$,
we obtain
\begin{equation*}
  x(t_{k+1}^+)< r(x(t_k^+)+M_0),
  \quad \text{for}\ k=1,2,\dots
\end{equation*}
By the comparison principle
applied to $x(t_k)$ and the sequence $\{y_k\}$
defined by $y_0=x(0)$ and $y_{k+1}=r(y_k+M_0)$, $k=1,2,\dots$,
\begin{equation}
  \limsup_{k\to \infty}x(t_k^+)
  \le \lim_{k\to \infty}y_k
  = \frac{rM_0}{1-r}.
  \label{ineq_xrM0}
\end{equation}
The existence of $M_2$ follows from \eqref{ineq_xM0} and \eqref{ineq_xrM0}.
\end{proof}

Define $\Omega_1$ to be the set of points $(s_1^0,s_2^0)$
such that, for some $x^0>0$,
the forward trajectory of the solution of \eqref{odes}
with initial value $(s_1^0,s_2^0,x^0)$
intersects $\Gamma^-$.
Then the boundary of $\Omega_1$
is the union of $\Gamma^-$
and the two lines of slope $R_{21}$
passing through $(\bar{s}_1,0)$ and $(0,\bar{s}_2)$, respectively.
Then, \begin{equation*}
  \Omega_1
  = \left\{
    (s_1,s_2): \begin{aligned}
      &s_1>\bar{s}_1 \;\text{or}\; s_2>\bar{s}_2,\;\text{and}\\
      &s_2 -\bar{s}_2 <R_{21}s_1,\; s_2> R_{21}(s_1-\bar{s}_1)
    \end{aligned}
  \right\}.
\end{equation*}

Define $\Omega_0$ to be the open set   complementary to $\Omega_1$
in the first quadrant   above and to
the right of 
$\Gamma^-$. Then, \begin{equation*}
  \Omega_0
  = \left\{
    (s_1,s_2): \begin{aligned}
      &s_1>\bar{s}_1 \;\text{or}\; s_2>\bar{s}_2,\;\text{and}\\
      &s_2 -\bar{s}_2 >R_{21}s_1\;\text{or}\; s_2< R_{21}(s_1-\bar{s}_1)
    \end{aligned}
  \right\}.
\end{equation*}

\begin{remark}
The sets $\Omega_0$ and $\Omega_1$ do not include the
marginal cases where $(s_1(0),s_2(0))$ lies on the lines of
slope $R_{21}$ passing through $(\bar{s}_1,0)$ or
$(0,\bar{s}_1)$. If $(s_1(0),s_2(0))$ lies on one of these lines
and $s_i(0)>0$ for $i=1,2$, then no impulses occur, since the
solution curve does not reach $\Gamma^-$ in finite time. If
$(s_1(0),s_2(0))=(\bar{s}_1,0)$ or
$(s_1(0),s_2(0))=(0,\bar{s}_2)$, then an impulse occurs immediately, and $(s_1(0^+),s_2(0^+))$ may be in either $\Omega_0$ or $\Omega_1$, depending on the location of $u^\IN$.
\end{remark}

In the following  case, the fermentation process fails.

\begin{lemma}
\label{lem_Omega0} 
If $(s_1(t),s_2(t),x(t))$ is a  solution of
\eqref{modeleq} with $(s_1(0),s_2(0))\in \Omega_0$,
then no impulses occur.
\end{lemma}

\begin{proof}
Since $\Omega_0$ is complementary to $\Omega_1$,
$(s_1(t),s_2(t))\notin\Gamma^-$ for any $t\ge 0$,
and hence no impulses occur.
\end{proof}

Next we define
a Lyapunov-type function \begin{equation}
V(s_1,s_2) =  Y_2(s_2^\IN-s_2)
-Y_1(s_1^\IN-s_1).
  \label{def_V}
\end{equation}
Then,
\begin{equation*}
  \Omega_1= \left\{
    (s_1,s_2): \begin{aligned}
      &s_1>\bar{s}_1 \;\text{or}\; s_2>\bar{s}_2,\;\text{and}\\
      &V(0,\bar{s}_2) < V(s_1,s_2) < V(\bar{s}_1,0)
    \end{aligned}
  \right\}
\end{equation*}
and \begin{equation*}
  \Omega_0= \left\{
    (s_1,s_2): \begin{aligned}
      &s_1>\bar{s}_1 \;\text{or}\; s_2>\bar{s}_2,\;\text{and}\\
      &V(s_1,s_2)< V(0,\bar{s}_2) \;\text{or}\; V(s_1,s_2)> V(\bar{s}_1,0)
    \end{aligned}
  \right\}.
\end{equation*}
Note that the level sets of $V$ are straight lines with slope $R_{21}$
in the $s_1$-$s_2$ plane.   For any fixed value of $s_1$, $V(s_1,s_2)$
is a decreasing function of $s_2$, and for any fixed value of $s_2$,
$V(s_1,s_2)$ is an increasing function of $s_1$. 

\begin{lemma}
Assume  $(s_1(t),s_2(t),x(t))$ is a  solution of
\eqref{modeleq}  with positive initial
conditions.
Let $t_0= 0$
and  $t_k$, $k\in \mathbb N$, denote the $k$th impulse time, if it exists;
otherwise set $t_k=\infty$.
Define $V(t)= V(s_1(t),s_2(t))$.
Then, \begin{enumerate}
  \item[(i)]
  $\frac{d}{dt}V(t)= 0$ for $t\in (t_k,t_{k+1})$,
  \item[(ii)]
  $V(t_{k}^+)= (1-r)V(t_{k}^-)$, if $t_k<\infty$.
\end{enumerate}
\label{lem_vdot}
\end{lemma}

\begin{proof}
(i) By the equations for $ds_i/dt$ in \eqref{odes}, \[
  \frac{d}{dt}V(t)
= \left(\frac{Y_2}{Y_2}-\frac{Y_1}{Y_1}\right)\min\{f_1(s_1(t)),f_2(s_2(t))\}x(t)= 0.
\]
(ii) Substituting \eqref{eq_ukplus}  into \eqref{def_V}, \[
  V(t_k^+)
  = (1-r)(s_1(t_k^-)-s_1^\IN,s_2(t_k^-)-s_2^\IN)\cdot (Y_1,-Y_2)
  =(1-r)V(t_k^-),
\] where $\cdot$ denotes the inner product in $\mathbb R^2$.
\end{proof}

By  the definition of $V(s_1,s_2)$ in \eqref{def_V} and Lemma \ref{lem_vdot},
the line \begin{equation}
  \begin{aligned}
  \{(s_1,s_2): V(s_1,s_2)=0\}
  =\{(s_1^\IN+vY_2,s_2^\IN+vY_1): v\in \mathbb R\}
  \end{aligned}
  \label{line_invariant}
\end{equation}
is invariant under \eqref{odes} for all $x(0)$. 
By symmetry,
we may assume the point $(\bar{s}_1,\bar{s}_2)$
lies on or above this invariant line,
i.e.,  $V(\bar{s}_1,\bar{s}_2)\le 0$,
or \begin{equation}
  s_2^\IN -\bar{s}_2 \le R_{21}(s_1^\IN-\bar{s}_1).
  \label{ineq_sbarR21}
\end{equation}

Next we show that  in the case that $(s_1^\IN,s_2^\IN)$
lies in $\Omega_0$,
once again the fermentation process
is doomed to fail.

\begin{lemma}
\label{lem_Omega0_sIN}
If $(s_1^\IN,s_2^\IN)\in\Omega_0$,
then every  solution of \eqref{modeleq}
with positive initial conditions has at most finitely many impulses.
\end{lemma}

\begin{proof}
Note that, $V(s_1^\IN,s_2^\IN)=0$,
so the condition $(s_1^\IN,s_2^\IN)\notin\Omega_1$
implies that either $V(0,\bar{s}_2)>0$ or $V(\bar{s}_1,0)<0$.
Since $V(0,\bar{s}_2)<V(\bar{s}_1,\bar{s}_2)\le 0$,
by  assumption \eqref{ineq_sbarR21},
  $V(\bar{s}_1,0)<0$.

We proceed using  
proof by contradiction.  Suppose that a  solution
$(s_1(t),s_2(t),x(t))$ of \eqref{modeleq}
with  positive initial conditions
has infinitely many impulses.
Denote the impulse times by $t_1<t_2<\cdots$.
By Lemma \ref{lem_vdot}(ii), $V(s_1(t_k^-),s_2(t_k^-))\to 0$ as
$k\to\infty$.  Hence, $V(s_1(t_k^-),s_2(t_k^-))>V(\bar{s}_1,0)$ for some $k\ge 1$.
By Lemma \ref{lem_Omega0}, no more impulses can occur.
\end{proof}

In the case of $(s_1^\IN,s_2^\IN)\in \Omega_1$,
under  assumption \eqref{ineq_sbarR21},
the line given by \eqref{line_invariant} intersects $\Gamma^-$
at the point $(\bar{s}_1,\widehat{s}_2)$ given by \[
  \widehat{s}_2= s_2^\IN- R_{21}(s_1^\IN-\bar{s}_1).
\]
Define $L$ to be 
the portion of the line given by \eqref{line_invariant}
from the point $(\bar{s}_1,\widehat{s}_2)$
to its image via the impulsive map, namely \begin{align*}
  L
  &=\{(s_1,s_2): \bar{s}_1\le s_1\le \bar{s}_1^+,\; V(s_1,s_2)=0\}\\
  &=\{(s_1,s_2): \bar{s}_1\le s_1\le \bar{s}_1^+,\; s_2=\widehat{s}_2+
  R_{21}(  s_1-\bar{s}_1)\}.
\end{align*}
Then $L$ is invariant under \eqref{modeleq} for all $x(0)$. 

\subsection{Existence of Periodic Orbits} \label{existence_PO}
Next we investigate under what conditions  the reactor has a periodic solution and the process has the potential to succeed.   

We regard
the emptying/refilling fraction $r\in (0,1)$ as a variable.
 Without loss of generality, from now on we assume that $(\bar{s}_1,\bar{s}_2)$ lies on
or above $L$.  Otherwise, from \eqref{ineq_sbarR21}, by symmetry we can relabel the resources. Therefore, the right endpoint of $L$ is $(\bar{s}_1^+(r),\widehat{s}_2^+(r))$ given by \[
  \bar{s}_1^+(r)= (1-r)\bar{s}_1+ rs_1^\IN
  \quad\text{and}\quad
  \widehat{s}_2^+(r)=
  (1-r)\widehat{s}_2+ rs_2^\IN.
\]
By \eqref{x_diff_s1},
the net change in $x(t)$ over one cycle with impulse at
$(\bar{s}_1,\widehat{s}_2)$ is \begin{equation}
  \mu(r)
  = Y_1\int_{\bar{s}_1}^{\bar{s}_1^+(r)}
  \left(1-\frac{D}{\min\{f_1(v),f_2(\widehat{s}_2+
  R_{21}(v-\bar{s}_1))\}}\right)dv.
  \label{def_mu}
\end{equation}

\begin{figure}[t]
\centering
{\includegraphics[trim = 1cm 0cm 1cm 0cm, clip, width=.49\textwidth]{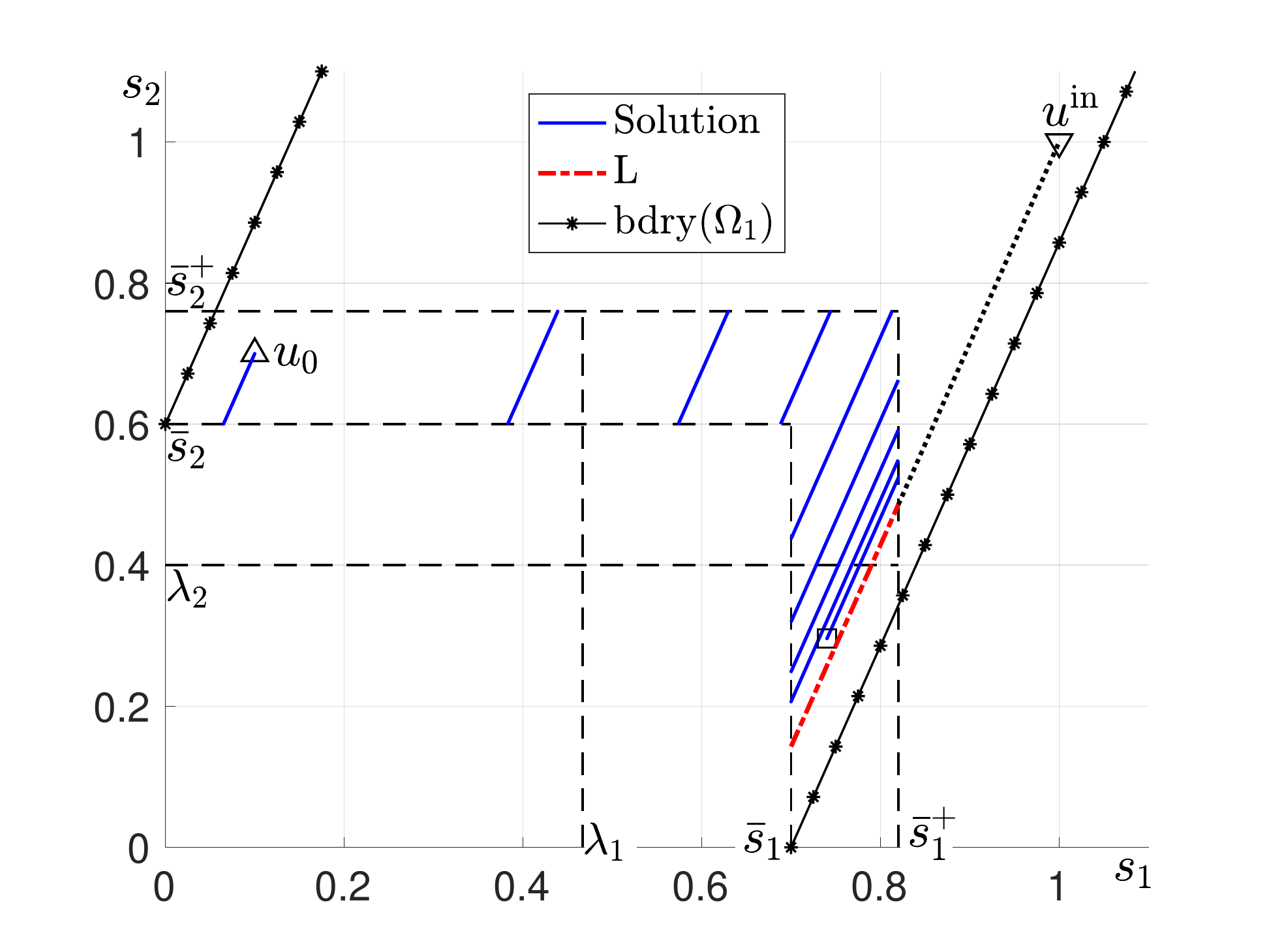}}\;
{\includegraphics[trim = 1cm 0cm 1cm 0cm, clip,width=.49\textwidth]{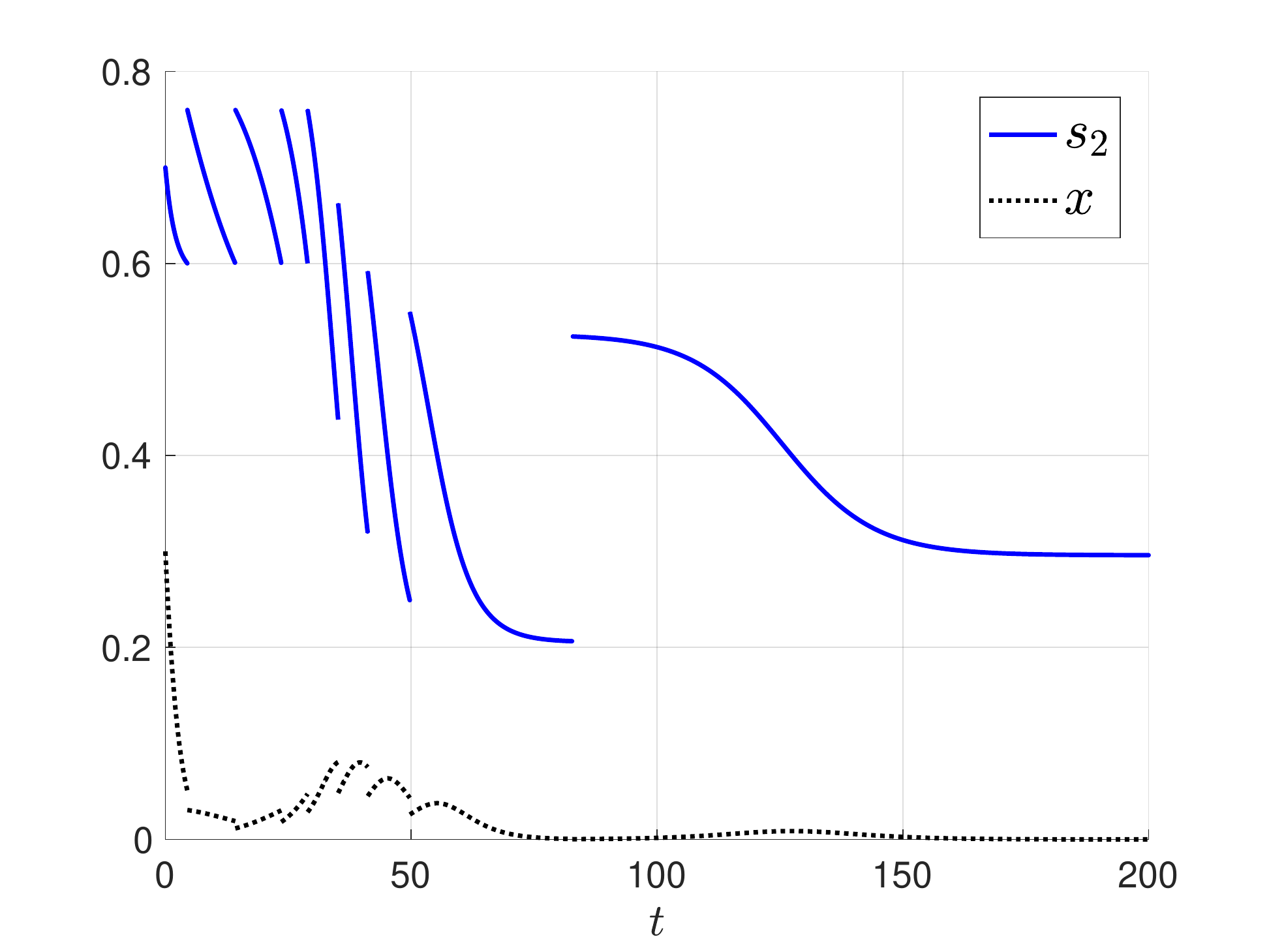}}
\par
\caption{%
No periodic solution exists,
and the $x$-component of every solution tends to $0$ as $t\to\infty$.
In the simulation,
the response functions are
$f_i(s_i)=\frac{m_is_i}{a_i+s_i}$, $i=1,2$,
with
$(m_1,m_2,a_1,a_2)=(2,2,1.4,1.2)$.
The parameters are
$(Y_1,Y_2,D)=(2,0.7,0.5)$
and
$(\bar{s}_1,\bar{s}_2,s_1^\IN,s_2^\IN,r)=(0.7,0.6,1,1,0.4)$.
The initial condition is $u_0=(s_1(0),s_2(0),x(0))=(0.1,0.7,0.3)$.
After a finite number of impulses, 
the orbit converges before reaching the threshold for an impulse, indicated by  $\square$.
The value of $\mu(r)\approx  -0.08<0$.
}
\label{fig_washout}
\end{figure}

We prove the following theorem concerning the existence and the uniqueness of periodic solutions.

\begin{theorem}
\label{thm_per_exists}
Assume $(s_1^\IN,s_2^\IN)\in \Omega_1$. If $r\in(0,1)$ and $\mu(r)>0$, then system (1) has a periodic orbit that is unique up to time translation and has one impulse per period.
On a periodic orbit,
$x(t_k^+)=\frac{(1-r)}{r}\mu(r)$ and
$x(t_k^-)=\frac{1}{r}\mu(r)$, for all 
 \ $k\in\mathbb{N}$.
 
 If $\mu(r)\le 0$, then system (1) has no periodic orbits.
 \end{theorem}

See Figure~\ref{fig_washout} for an illustration 
of the case with $\mu(r)<0$.

\begin{proof}
Suppose there is a positive periodic solution $(s_1(t),s_2(t),x(t))$.
Since \eqref{odes} has no periodic orbits,
the solution has at least one impulse.
By periodicity there are infinitely many impulses.
Denote the impulse times by $t_1<t_2<\cdots$
and the number of impulses within a period by $N$.
Then $x(t_{N+k}^\pm)=x(t_k^\pm)$ for all $k\in\mathbb N$.
By \eqref{modeleq},
\begin{align*}
  x(t_{k+1}^-)= x(t_k^+)+ \mu(r)
  \quad\text{and}\quad
  x(t_k^+)= (1-r)x(t_k^-).
\end{align*}
Thus, \begin{align}
  x(t_{k+1}^-)= (1-r)x(t_k^-)+ \mu(r).
  \label{x_iter_m}
\end{align}
If $x(t_1^-)>\mu(r)$ (resp.\ $x(t_1^+)<\mu(r)$),
then from \eqref{x_iter_m}
it can be shown by induction
that $x(t_k^-)$ is a  strictly decreasing (resp.\
strictly increasing) sequence,
contradicting  $x(t_{N+k}^-)=x(t_k^-)$.  
Hence, on any periodic orbit there is only one impulse, and it
follows from  \eqref{x_iter_m}, on a periodic orbit   $x(t_k^-)=\frac{1}{r}\mu(r)$ for all $k\in\mathbb N$.
Therefore,    $\mu(r)>0$, since the solution must be positive, and if a
periodic orbit exists it is unique. 

If $\mu(r)>0$,
the solution of \eqref{modeleq}
with initial condition $(\bar{s}_1^+,\widehat{s}_2^+,\frac{1-r}{r}\mu(r))$
is periodic, since if$x(t_{k+1}^-)=\frac{1}{r}\mu(r)$, then
$x(t_{k+1}^+)=\frac{1-r}{r}\mu(r)$ and if
$x(t_k^+)=\frac{1-r}{r}\mu(r)$, then
$x(t_{k+1}^-)=\frac{1}{r}\mu(r)$. 
\end{proof}

\begin{proposition}
If $\mu(1)>0$, then there exists a unique $r^*\in [0,1)$ such that $\mu(r)>0$ for all $r\in(r^*,1]$ and $\mu(r)\leq 0$ for all $r\in(0,r^*].$
\end{proposition}
\begin{proof}
Let \begin{equation}
  r_*= \max\{r\in [0,1]: \mu(v)\le 0\;\;\forall\;v\in [0,r]\}.
  \label{def_rstar}
\end{equation}
Note that $r_*$ is well-defined, since $\mu(r)$ is continuous and $\mu(0)=0$.
By definition, $\mu(r)\le 0$ for all $r\in[0,r_*]$.
 Since $\mu(r)$ is continuous, if $\mu(1)>0$ then $r_*<1$. Furthermore, 
  $\mu(r)>0$ for all $r\in (r_*,1]$, since
  $f_1(s_1)$ and $f_2(s_2)$ are monotone increasing, and so 
the integrand in \eqref{def_mu} with  $v=\bar{s}_1^+(r_*)$
must be positive. Otherwise, $r_*$ can be increased,
contradicting  definition \eqref{def_rstar}.
By the monotonicity of $f_1(s_1)$ and $f_2(s_2)$,
the integrand in \eqref{def_mu}
with $v=\bar{s}_1^+(r)$ remains positive
for $r_*<r   \le 1$.  Hence, $\mu(r)>0$ for $r\in (r_*,1]$.   
\end{proof}

\begin{remark}
If $\lambda_1\le \bar{s}_1$ and $\lambda_2\le \widehat{s}_2$,
then $\mu(r)>0$ for all $r\in (0,1)$, i.e.,\ $r_*=0$, because the integrand in \eqref{def_mu}
is then positive for all $v\in (\bar{s}_1,s^\IN)$.
\end{remark}

\subsection{Global Stability of Periodic Orbits}
\label{subsec_per_gas}
In this subsection we fix an $r\in (r_*,1)$,
where $r_*$ is the number given in Theorem \ref{thm_per_exists}.
Hence, $\mu(r)>0$
and a unique periodic orbit exists.

For each point $(s_1,s_2)\in \Omega_1$,
we denote by $\pi^-(s_1,s_2)$,
the point of intersection of 
the line through $(s_1,s_2)$ of slope $R_{21}$ with $\Gamma^-$, i.e., \[
  \pi^-(s_1,s_2)
  = \begin{cases}
    (s_1+ R_{12}(\bar{s}_2-s_2),\bar{s}_2),
    &\text{if}\;\; V(s_1,s_2)\le V(\bar{s}_1,\bar{s}_2),\\
    (\bar{s}_1,s_2+ R_{21}(\bar{s}_1-s_1)),
    &\text{if}\;\; V(s_1,s_2)\ge V(\bar{s}_1,\bar{s}_2).
  \end{cases}
\]

For each point $(s_1,s_2)\in \Gamma^-$,
we denote by $\pi^+(s_1,s_2)$
the pre-image of $(s_1,s_2)$ in $\Gamma^+$ under $\pi^-$.
Hence, $\pi^+(s_1,s_2)\in \Gamma^+$ satisfies $\pi^-(\pi^+(s_1,s_2))= (s_1,s_2)$
for all $(s_1,s_2)\in \Gamma^-$.
Let $g: \Gamma^- \to \Gamma^+$, be the image of $(s_1,s_2) \in\Gamma^-$ under the impulsive map,
i.e., \[
  g(s_1,s_2)= (1-r)(s_1,s_2)+ r(s_1^\IN,s_2^\IN).
\]

For each point $(s_1^0,s_2^0)\in \Omega_1$,  let
$I(s_1^0,s_2^0)$ denote
the net change in $x(t)$,
over the time interval from $t=0$ to the first impulse time.
Therefore, for  a solution $(s_1(t),s_2(t),x(t))$ of \eqref{modeleq}
satisfying $(s_1(0),s_2(0))=(s_1^0,s_2^0)$
that has at least one impulse,  by Lemma \ref{lem_odes}, 
\begin{equation*}
  \begin{aligned}
  &I(s_1^0,s_2^0)\\
  &=\begin{cases}
      Y_2\int_{\bar{s}_2}^{s_2^0}
      \left(1-\frac{D}{\min\{f_1(s_1^0+ R_{12}(v-s_{2}^0),f_2(v)\}}\right)dv,
      &\text{if}\;\; V(s_1,s_2)\le V(\bar{s}_1,\bar{s}_2),\\
      Y_1\int_{\bar{s}_1}^{s_1^0}
      \left(1-\frac{D}{\min\{f_1(v),f_2(s_2^0+ R_{21}(v-s_{1}^0)\}}\right)dv,
      &\text{if}\;\; V(s_1,s_2)\ge V(\bar{s}_1,\bar{s}_2).
  \end{cases}
  \end{aligned}
\end{equation*}
Note that 
$\mu(r)=I(\pi^+(\bar{s}_1,\widehat{s}_2))$. 

If $(s_1(t),s_2(t),x(t))$ is a solution of \eqref{odes}
with $(s_1(t_0),s_2(t_0))= (s_1^0,s_2^0) \in \Gamma^+$
and $(s_1(t_1),s_2(t_1))= (s_1^1,s_2^1) \in \Gamma^-$ for some $t_0<t_1$,
 then  $(s_1^0,s_2^0)= \pi^+(s_1^1,s_2^1)$ and the net change in $x(t)$ over the time interval $[t_0,t_1]$
is $I(\pi^+(s_1^1,s_2^1))$.
Since $V(\bar{s}_1,\bar{s}_2)<V(\bar{s}_1^+,\bar{s}_2^+)<V(s_1^\IN,s_2^\IN)$,
from $V(s_1^\IN,s_2^\IN)=V(\bar{s}_1,\widehat{s}_2)$
there exists $\widetilde{s}_2\in (\widehat{s}_2,\bar{s}_2)$
such that $V(\bar{s}_1,\widetilde{s}_2)=V(\bar{s}_1^+,\bar{s}_2^+)$.
Define \begin{equation}
    \Gamma^-_A
  = \{(\bar{s}_1,s_2): 0<s_2\le\widetilde{s}_2: I(\pi^+(s_1,s_2))>0\}.
  \label{def_Gamma1A}
\end{equation}

\begin{lemma}
\label{lem_GammaA}
Assume $(s_1^\IN,s_2^\IN)\in \Omega_1$ 
and $\mu(r)>0$.
Then, \begin{equation}
  \Gamma^-_A
  = \{\bar{s}_1\}\times (s_{2\sharp},\widetilde{s}_2]
  \label{GammaA_s2sharp}
\end{equation}
for some $s_{2\sharp}\in (0,\widehat{s}_2)$.
\end{lemma}
\begin{proof}
By  assumption \eqref{ineq_sbarR21},
for any $s_2\in [0,\widetilde{s}_2]$, \begin{equation*}
  I(\pi^+(\bar{s}_1,s_2))
  =Y_1\int_{\bar{s}_1}^{\bar{s}_1^+}
    \left(1-\frac{D}{\min\{f_1(v),f_2(s_2+
R_{21}(v-\bar{s}_{1})\}}\right)dv.
\end{equation*}
Let \[
  \Lambda
  =\{s_2\in (0,\widetilde{s}_2): I(\pi^+(\bar{s}_1,s_2))> 0\}.
\]
By the monotonicity of $f_1(s_1)$ and $f_2(s_2)$,
since $I(\pi^+(\bar{s}_1,\widehat{s}_2))=\mu(r)>0$,
$\Lambda$
is an interval containing $\widehat{s}_2$ with right endpoint $\widetilde{s}_2$.
Hence, $\Gamma^-_A$
takes the form \eqref{GammaA_s2sharp}
for some $s_{2\sharp}\in (0,\widehat{s}_2)$.
\end{proof}

Let $\Omega_{1A}$ be the set of points $(s_1^0,s_2^0)$
such that, for some $x^0>0$,
the forward trajectory of the solution of \eqref{odes}
with initial value $(s_1^0,s_2^0,x^0)$
passes through $\Gamma^-_A$.
Then, \begin{equation}
  \Omega_{1A}
  = \{(s_1,s_2)\in \Omega_1: \pi^-(s_1,s_2)\in \Gamma^-_A\}.
  \label{def_Omega1A}
\end{equation}
In the case $\mu(r)>0$,
by \eqref{GammaA_s2sharp}, \begin{equation}
  \Omega_{1A}
  = \{(s_1,s_2)\in \Omega_1: V_- < V(s_1,s_2) < V_+\}
  \label{eq_Omega1A_Vbar},
\end{equation}
where $V_-=V(\bar{s}_1,\widetilde{s}_2)$
and $V_+=V(\bar{s}_1,s_{2\sharp})$.

\begin{lemma}
\label{lem_Omega1A}
Assume $(s_1^\IN,s_2^\IN)\in \Omega_1$ 
and $\mu(r)>0$.
Let $(s_1(t),s_2(t),x(t))$ be a  solution of
\eqref{modeleq} with $x(0)>0$ and 
$$
  (s_1(0),s_2(0))\in \Omega_{1A}.
$$
 The solution converges to the unique periodic solution
given by Theorem~\ref{thm_per_exists},if and only if  $x(0)>
-I(s_1(0),s_2(0))$. If $x(0) \le -I(s_1(0),s_2(0))$,  then no impulses occur. 
\end{lemma}

\begin{proof}
If $x(0)\le   -I(s_1(0),s_2(0))$,
then by Lemma \ref{lem_odes}
the value of $x(t)$ approaches $0$
before any impulses occur.

If $x(0)>  -I(s_1(0),s_2(0))$,
then the first impulse occurs at some finite time $t_1>0$.
The condition $(s_1(0),s_2(0))\in \Omega_{1A}$
implies that $(s_1(t_1^+),s_2(t_1^+))\in \Omega_{1A}$.
Hence, $I(s_1(t_1^+),s_2(t_1^+))>0$.
This implies  that
the net change of $x(t)$ is positive
over any time interval from $t_1$ to a time before the next impulse.
Hence, another impulse occurs at some finite time $t_2>t_1$.
Inductively, it follows that impulses occur indefinitely.
By Lemma \ref{lem_vdot}, $\lim_{t\to \infty}V(s_1(t),s_2(t))= 0$.
Hence, \[
  (s_1(t_k^-),s_1(t_k^-))
  \to (\bar{s}_1,\widehat{s}_2)
  \quad\text{and}\quad
  (s_1(t_k^+),s_1(t_k^+))
  \to (\bar{s}_1^+,\widehat{s}_2^+).
\]
Therefore, by Lemma \ref{lem_odes}
and the relation $I(\pi^+(\bar{s}_1,\widehat{s}_2))=\mu(r)$, \[
  \lim_{k\to\infty}\big( x(t_{k+1}^-)-x(t_k^+)\big)
  = \mu(r).
\] On the other hand, the impulsive map in \eqref{modeleq}
gives $\lim_{k\to \infty}x(t_{k}^+)-(1-r)x(t_k^-)= 0$.
This gives \[
  \lim_{k\to \infty}\big(x(t_{k+1}^-)-(1-r)x(t_k^-)\big)= \mu(r),
\]
which implies $\lim_{k\to \infty}x(t_{k}^-)=\frac{1}{r}\mu(r)$
and $\lim_{k\to \infty}x(t_{k}^+)=\frac{1-r}{r}\mu(r)$.
We conclude that the solution converges to the periodic orbit.
\end{proof}

\begin{corollary}\label{cor:siIN}
If $(s_1^\IN,s_2^\IN)\in\Omega_1$
and $\mu(r)>0$,
then all solutions to \eqref{modeleq}
with $(s_1(0),s_2(0))=(s_1^\IN,s_2^\IN)$ and $x(0)>0$ converge to the periodic orbit
given by Theorem \ref{thm_per_exists}.
\end{corollary}

\begin{proof}
By the definitions of   $\Gamma_A^-$ and  $\Omega_{1A}$ given
in \eqref{def_Gamma1A} and \eqref{def_Omega1A}, respectively, 
$(s_1^\IN,s_2^\IN)\in \Omega_{1A}$.
Since $I(s_1^\IN,s_2^\IN)>I(\pi^+(\bar{s}_1,\widehat{s}_2))=\mu(r)>0$,
the desired result follows from Lemma \ref{lem_Omega1A}.
\end{proof}

For each $(s_1^0,s_2^0)\in \Omega_1$,
let $N_0=N_0(s_1^0,s_2^0)$ be the smallest positive integer
such that \begin{equation}
  (g\circ \pi^-)^{N_0}(s_1^0,s_2^0)\in \Omega_{1A}.
  \label{def_N0}
\end{equation} 
In particular, $N_0(s_1^0,s_2^0)=1$ for all $(s_1^0,s_2^0)\in \Omega_{1A}$.
For $(s_1^0,s_2^0)\in \Omega_1\setminus\Omega_{1A}$,
by the identities 
$V(g(s_1,s_2))= 
(1-r)
V(s_1,s_2)$ and $V(\pi^-(s_1,s_2))=V(s_1,s_2)$,
\begin{equation}
  V\big((g\circ \pi^-)^{n}(s_1^0,s_2^0)\big)
  = (1-r)^n
  V(s_1^0,s_2^0),
  \quad n=1,2,\cdots.
  \label{eq_Vgn}
\end{equation}
If $V(s_1^0,s_2^0)<0$, then by \eqref{eq_Omega1A_Vbar},
the condition $(s_1^0,s_2^0)\notin \Omega_{1A}$
is equivalent to $V(s_1^0,s_2^0)\le V_-$.
Thus,  by \eqref{eq_Vgn},  condition
\eqref{def_N0}, is equivalent to \[
  (1-r)^{N_0}V(s_1^0,s_2^0)> V_-.
\]
Similarly, if $V(s_1^0,s_2^0)>0$ and $(s_1(0),s_2(0))\in \Omega_{1A}$,
then   condition \eqref{def_N0} is equivalent to \[
  (1-r)^{N_0}V(s_1^0,s_2^0)< V_+.
\] 
Hence, 
\begin{equation}
  N_0(s_1^0,s_2^0)
  = \begin{cases}
    \left \lceil \frac{\ln(V(s_1^0,s_2^0)/V_-)}{-\ln(1-r)} \right\rceil,
    &\text{if}\;\; V(s_1^0,s_2^0)<V_-,\\[.5em]
    \left \lceil \frac{\ln(V(s_1^0,s_2^0)/V_+)}{-\ln(1-r)} \right\rceil,
    &\text{if}\;\; V(s_1^0,s_2^0)>V_+,
  \end{cases}
  \label{eq_N0}
\end{equation}
where $\lceil y\rceil$ is the least integer greater than or equal to $y$.

For any solution $(s_1(t),s_2(t),x(t))$ of \eqref{modeleq}
with  $(s_1(0),s_2(0))=(s_1^0,s_2^0)\in\Omega_{1}$, and
$x(0)>0$,
$$x(t_1^-)=x(0)+I(s_1^0,s_2^0).$$
Note that $x(t_k^+)= (1-r)x(t_k^-)$ by the impulsive map in \eqref{modeleq},
and $x(t_{k+1}^-)= x(t_k^+)+ I(s_1(t_k^+),s_2(t_k^+))$ by Lemma \eqref{lem_odes}.
Hence, for any $n=2,3,\dots,$
the left limit of $x(t)$
at the $n$th impulse, if it exists, equals \[
  x(t_n^-)
  = (1-r)x(t_{n-1}^-)+ I((g\circ\pi^-)^{n-1}(s_1^0,s_2^0)),
\]
where $t_0=0$, and $t_k$, $k\ge 1$, is the $k$th impulse time.
By induction, \[
  x(t_n^-)
  =(1-r)^{n-1}x(0)
  + \sum_{k=1}^{n} (1-r)^{n-k}I((g\circ\pi^-)^{k-1}(s_1^0,s_2^0)).
\]
Thus the condition $x(t_n^-)>0$ is equivalent to \[
  x(0)> -\sum_{k=1}^{n} (1-r)^{-(k-1)}I((g\circ\pi^-)^{k-1}(s_1^0,s_2^0)).
\]
We define  $X(s_1^0,s_2^0)$ to be 
the least value so that if  $x(0)>X(s_1^0,s_2^0)$
then $(s_1(t_*^-),s_2(t_*^-))\in \Gamma^-_A$ for some $t_*>0$.
Hence, 
\begin{equation}
  X(s_1^0,s_2^0)
  = -\left(
  \min_{1\le n\le N_0(s_0,s_1)}
  \sum_{k=1}^n
  (1-r)^{-(k-1)}I\Big( (g\circ \pi^-)^{k-1}(s_1^0,s_2^0) \Big)
  \right).
  \label{eq_X}
\end{equation}
In particular,
\[
  X(s_1^0,s_2^0)= 
  -I(s_1^0,s_2^0)
  \quad\text{if}\; \  N_0(s_1^0,s_2^0)= 1.
\]

The following proposition extends Lemma \ref{lem_Omega1A}.

\begin{proposition}
\label{prop_Omega1B}
Assume $(s_1^\IN,s_2^\IN)\in \Omega_1$
and $\mu(r)>0$.
Let $(s_1(t),s_2(t),x(t))$ be a  solution of \eqref{odes}
with $(s_1(0),s_2(0))\in \Omega_1$  and
$x(0)>0$.
\begin{enumerate}
\item[(i)]
If $x(0)\le X(s_1(0),s_2(0))$, then there are at most $N_0(s_1^0,s_2^0)-1$ impulses.
\item[(ii)]  If $x(0)> X(s_1(0),s_2(0))$, then the solution converges to
the unique periodic orbit given by Theorem \ref{thm_per_exists}.
\end{enumerate}
\end{proposition}

\begin{proof}
(i) Suppose $x(0)\le X(s_1(0),s_2(0))$
and the solution has at least $N_0=N_0(s_1^0,s_2^0)$ impulses.
Denote the first $N_0$  impulse times by $t_1<t_2<\cdots<t_{N_0}$.
Then, by Lemma \ref{lem_odes} and the definition of $X(s_1,s_2)$,
for some $k\in \{1,\cdots,N_0\}$, \[
  x(t_k^-)= x(0)- X(s_1(0),s_2(0))\le 0,
\] contradicting   the positivity of the solution.

(ii) If $x(0)> X(s_1(0),s_2(0))$,
then the solution has at least $N_0$ impulses.
Denote the $N_0$th impulse time by $t_{N_0}$.
Then, \[
  (s_1(t_{N_0}^+),s_2(t_{N_0}^+))
  = (g\circ \pi^-)^{N_0}(s_1(0),s_2(0))\in \Omega_{1A}.
\] Since $(s_1(t_{N_0}^+),s_2(t_{N_0}^+))\in \Gamma^+$,
by the definition of $\Omega_{1A}$,  $I(s_1(t_{N_0}^+),s_2(t_{N_0}^+))>0$.
Hence, the result follows from
Lemma \ref{lem_Omega1A}.
\end{proof}

\begin{example}
Consider \eqref{modeleq} with the Monod functional responses
$f_i(s_i)=\frac{m_is_i}{a_i+s_i}$, $i=1,2$,
and parameters 
$(m_1,m_2,a_1,a_2)=(2,2,1.9,0.3)$,
$(Y_1,Y_2,D)=(4,1.9,0.5)$,
and
$(\bar{s}_1,\bar{s}_2,s_1^\IN,s_2^\IN,r)=(0.6,0.5,1,1,0.4)$.

We compute the following quantities using their definition.
\[
  (\bar{s}_1^+,\bar{s}_2^+)
  = (0.76,0.7),\;
  \widetilde{s}_2\approx 0.36,\;
  \widehat{s}_2\approx 0.16,\;
  \mu(r)\approx 0.03,\;
  V_-\approx -0.39.
\]
Taking the initial values $(s_1^0,s_2^0)=(0.23,0.6)$,  we have
$V(s_1^0,s_2^0)=-2.32$.
Then, \[
  N_0(s_1^0,s_2^0)
  = 
  \left\lceil
  \frac{\ln(V(s_1^0,s_2^0)/V_-)}{-\ln(1-r)}
  \right\rceil
  =
  \left\lceil
  3.4908
  \right\rceil
  = 4.
\]
The approximated values of $I((g\circ \pi^-)^n(s_1^0,s_2^0))$, $1\le n<N_0$,
are as follows.
\[\begin{array}{c|c|c|c|c}
  n& 0& 1& 2& 3
  \\
  \hline
  (1-r)^{-n}I((g\circ \pi^-)^n(s_1^0,s_2^0))
  & -0.2970
  & -0.1785
  & -0.0441
  & 0.0846
\end{array}\]
Therefore $X(s_1^0,s_2^0)\approx 0.2970+0.1785+0.0441= 0.5196$.

In Figures~\ref{fig_failed} and \ref{fig_cycling}, 
the initial data satisfies $x(0)=0.5< X(s_1^0,s_2^0)$
and $x(0)=0.53> X(s_1^0,s_2^0)$, respectively.
By Proposition \ref{prop_Omega1B}
the fermentation succeeds only in the latter case.
\label{ex_N0}
\end{example}

If $(s_1^\IN,s_2^\IN)\in \Omega_1$ 
and $\mu(r)\le 0$,
then, by Theorem \ref{thm_per_exists},
system \eqref{modeleq} has no periodic solution.
The following proposition
asserts that the fermentation fails in this case.

\begin{proposition}
Assume $(s_1^\IN,s_2^\IN)\in \Omega_1$.
\begin{enumerate}
\item[(i)]
If $\mu(r)< 0$,
then for every solution of \eqref{modeleq}
with positive initial conditions,
only finitely many impulses occur.
\item[(ii)]
If $\mu(r)= 0$,
then for every solution of \eqref{modeleq}
with positive initial conditions,
either only finitely many impulses occur,
or the time between impulses tends to infinity.
\end{enumerate}
\label{prop_I0}
\end{proposition}

\begin{proof}
Let $(s_1(t),s_2(t),x(t))$ be a solution of \eqref{modeleq}
with positive initial conditions.
Suppose the solution has infinitely many impulses.
Denote the impulse times by $t_1<t_2<\cdots$.
Then by Lemma \ref{lem_vdot},
$(s_1(t_k^-),s_2(t_k^-))\to (\bar{s}_1,\widehat{s}_2)$
and 
$(s_1(t_k^+),s_2(t_k^+))\to (\bar{s}_1^+,\widehat{s}_2^+)$
as $k\to \infty$.

(i) In the case $\mu(r)< 0$,
by Lemma \ref{lem_odes}
and the relation $I(\pi^+(\bar{s}_1,\widehat{s}_2))=\mu(r)$, \[
  \lim_{k\to\infty}\big( x(t_{k+1}^-)- x(t_k^+)\big)
  = \mu(r).
\] 
(ii) On the other hand, the impulsive map in \eqref{modeleq}
gives $x(t_{k}^+)= (1-r)x(t_k^-)$.
This implies $\lim_{k\to\infty}x(t_k^-)=\frac{1}{r}\mu(r)<0$,
contradicting to the positivity of the solution.

In the case $\mu(r)= 0$,
by Lemma \ref{lem_odes}
and the relation $I(\pi^+(\bar{s}_1,\widehat{s}_2))=\mu(r)=0$, \[
  \lim_{k\to \infty}\big(x(t_{k+1}^-)-x(t_k^+)\big)
  = 0.
\] 
By the relation $x(t_{k}^+)= (1-r)x(t_k^-)$,
it follows that $\lim_{k\to \infty}x(t_k^\pm)=0$.
Hence, the trajectory of $(s_1(t),s_2(t),x(t))$,
$t\in (t_{k},t_{k+1})$,
approaches the heteroclinic orbit of \eqref{odes}
from $(\bar{s}_1^+,\widehat{s}_2^+,0)$
to $(\bar{s}_1,\widehat{s}_2,0)$.
This implies $\lim_{k\to \infty}\big(t_{k+1}-t_k\big)= \infty$.
\end{proof}

By \eqref{eq_N0}, the function $N_0(s_1,s_2)$
of $(s_1,s_2)\in \Omega_1$
has an upper bound \begin{equation*}
  \bar{N}= \max\{N_0(0,\bar{s}_2),N_0(\bar{s}_1,0)\}.
\end{equation*}
We summarize our results as follows.

\begin{theorem}\label{Theorem_summary}
Consider system \eqref{modeleq}.
\begin{enumerate}[(i)]
\item
If $(s_1^\IN,s_2^\IN)\in \Omega_0$,
then every solution
has at most finitely many impulses.
\item
If $(s_1^\IN,s_2^\IN)\in \Omega_1$
and $\mu(r)\le 0$, 
then the fermentation fails in the sense
that for every solution
with positive initial conditions,
either only finitely many impulses occur,
or the time between impulses tends to infinity.
\item
If $(s_1^\IN,s_2^\IN)\in \Omega_1$ 
and $\mu(r)>0$,
then there is a unique periodic orbit.
Moreover, for any  solution
$(s_1(t),s_2(t),x(t))$, with positive initial conditions, 
the number of impulse times is either infinite
or is less than $\bar{N}$.
The case with infinitely many impulses occurs if and only if \[
  (s_1(0),s_2(0))\in \Omega_1
  \quad\text{and}\quad
  x(0)> X(s_1(0),s_2(0)).
\]
\end{enumerate}
\label{thm_X}
\end{theorem}

\begin{proof}
The theorem follows from
Lemmas \ref{lem_Omega0} and \ref{lem_Omega0_sIN},
and Propositions \ref{prop_Omega1B} and \ref{prop_I0}.
\end{proof}

\begin{figure}[htbp]
\centering
{\includegraphics[trim = 1cm 0cm 1cm 0cm, clip, width=.49\textwidth]{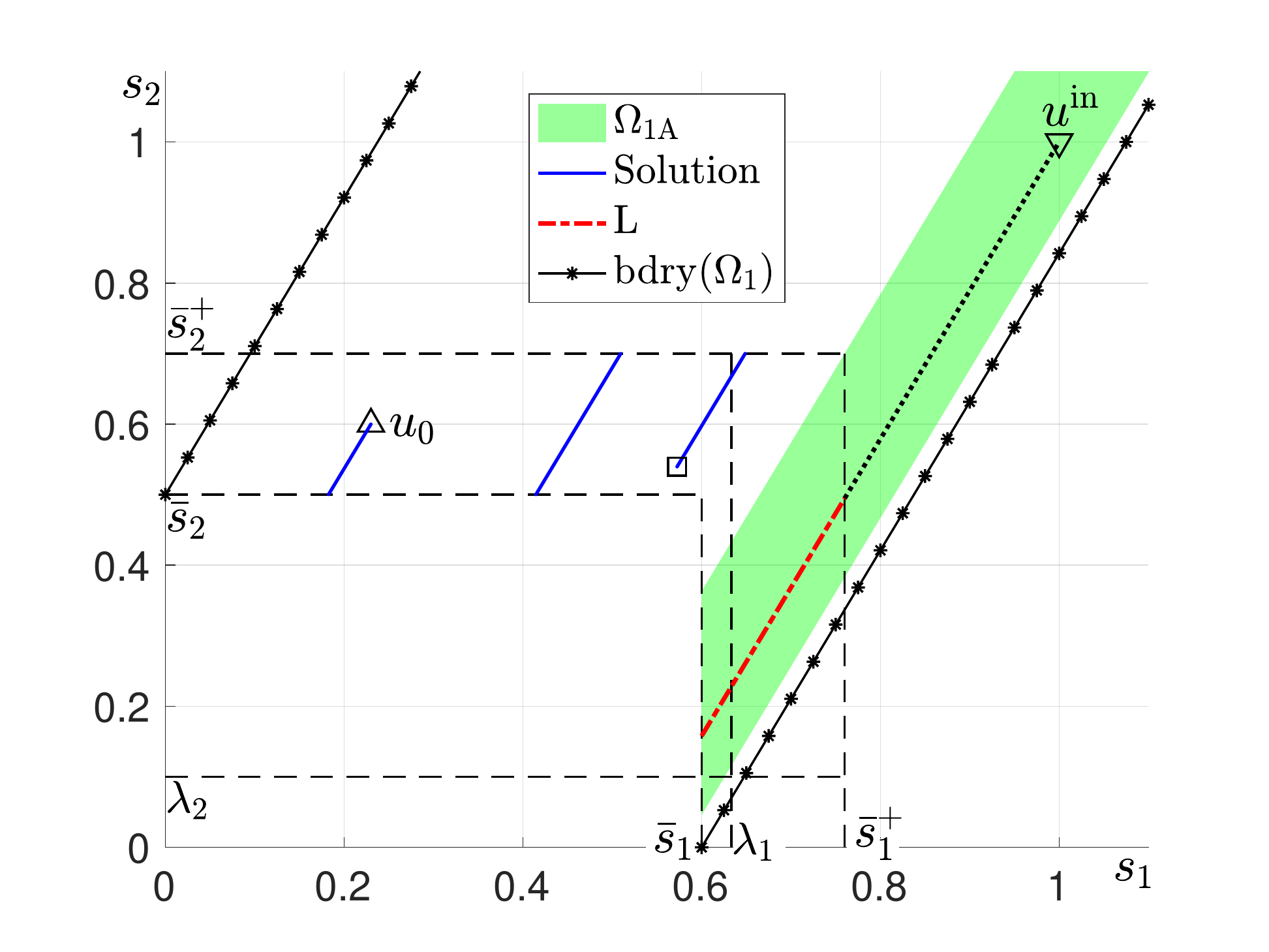}}
{\includegraphics[trim = 1cm 0cm 1cm 0cm, clip, width=.49\textwidth]{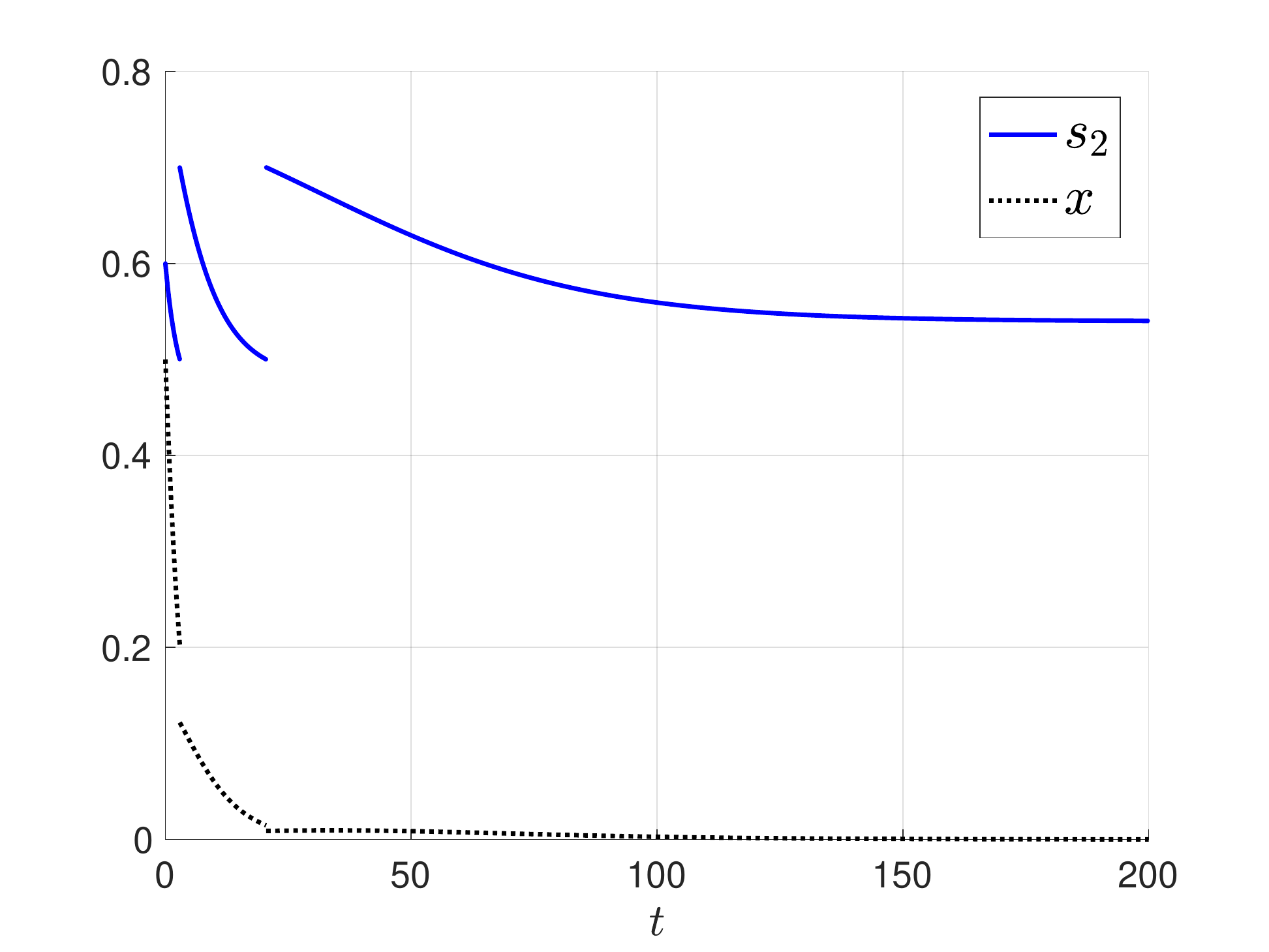}}
\caption{%
If $x(0)\le X(s_1(0),s_2(0))$,
then only finitely many impulses occur.
The orbit converges, indicated by  $\square$,
after a finite number of impulses,
and the $x$-component of the solution tends to $0$ as $t\to\infty$.
The parameters are
the values given in Example \ref{ex_N0},
and the initial condition is $(s_1(0),s_2(0),x(0))=(0.23,0.6,0.5)$.}
\label{fig_failed}
\end{figure}
\begin{figure}[htbp]
{\includegraphics[trim = 1cm 0cm 1cm 0cm, clip, width=.49\textwidth]{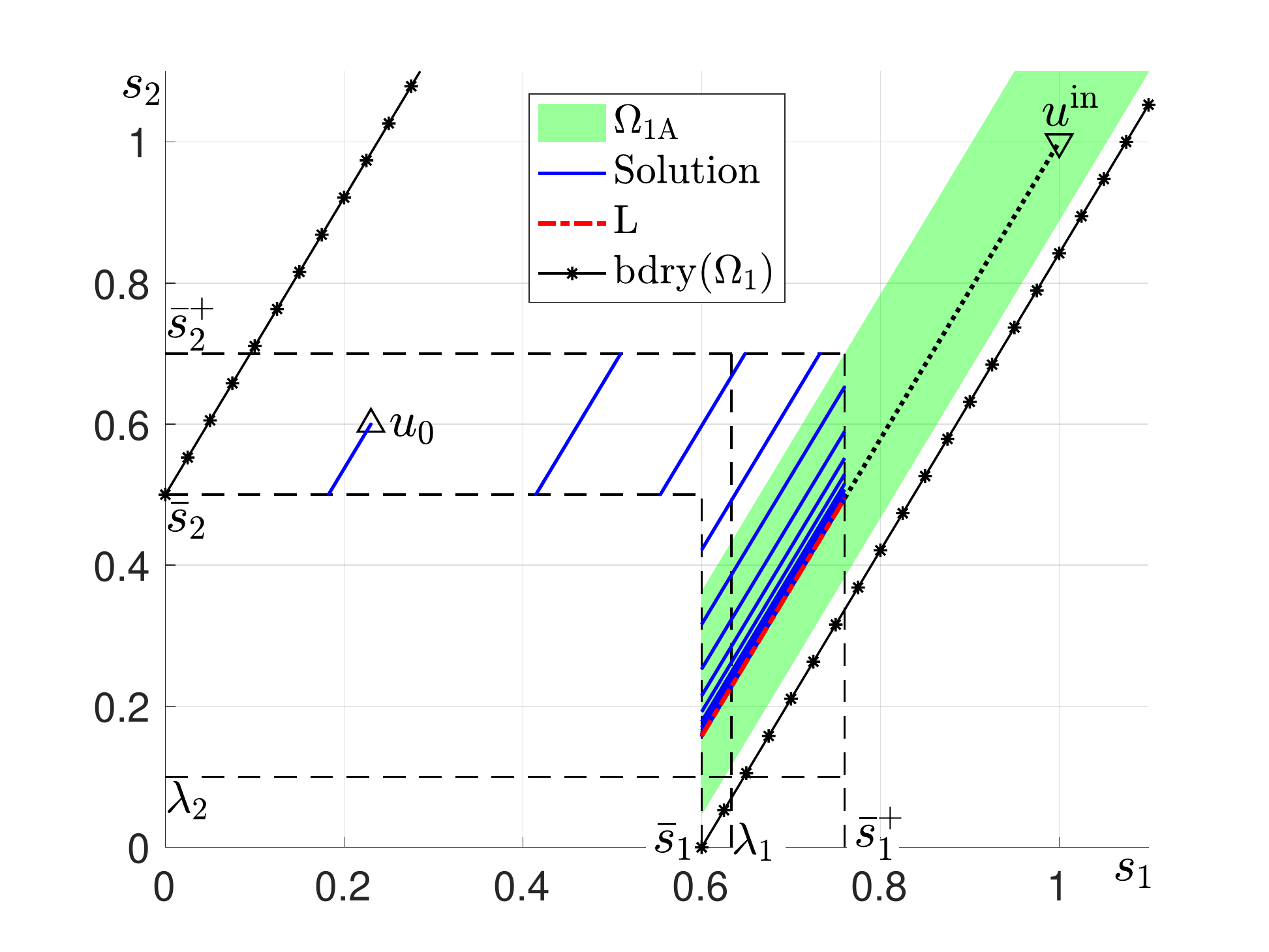}}
{\includegraphics[trim = 1cm 0cm 1cm 0cm, clip, width=.49\textwidth]{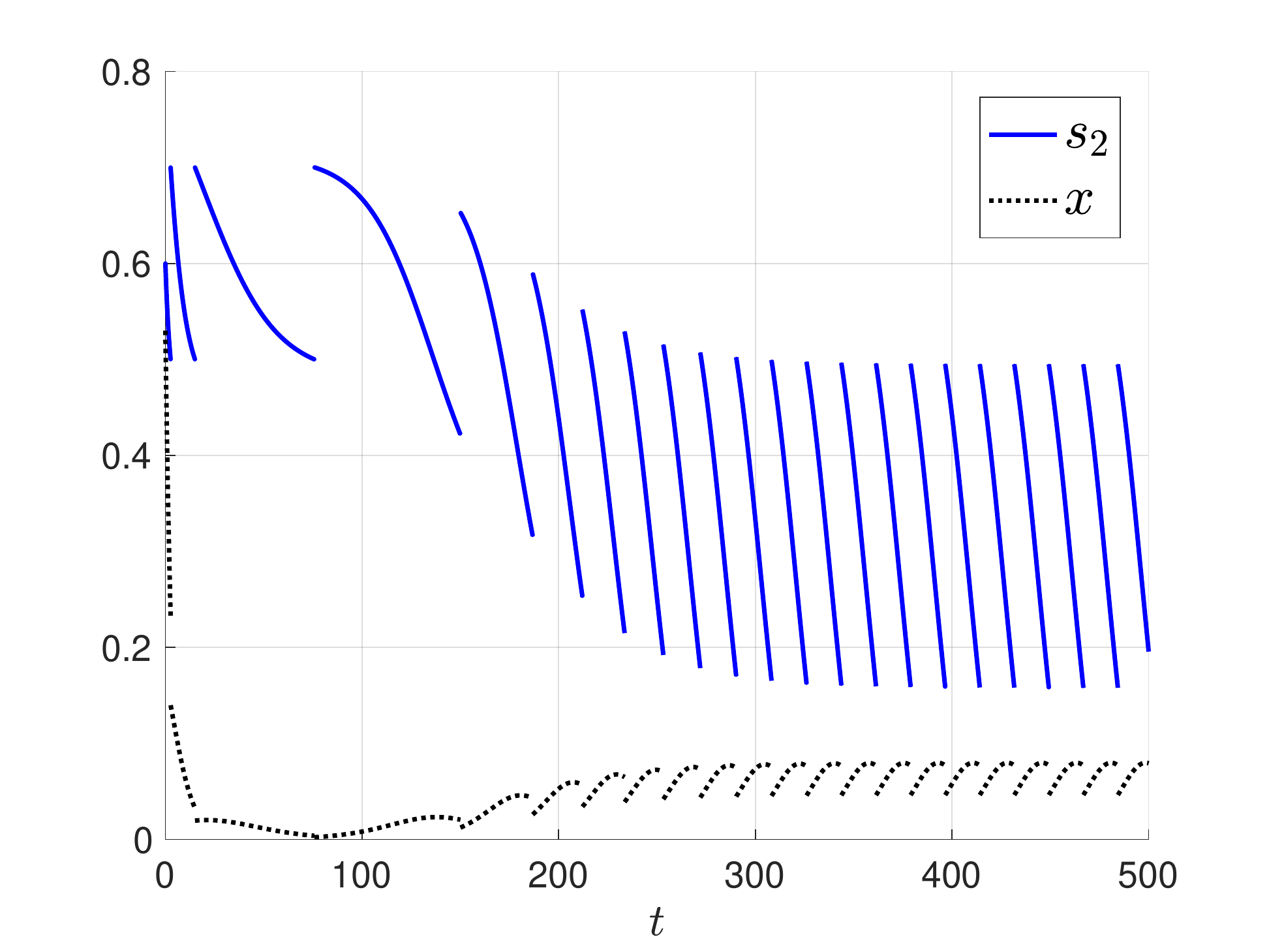}}
\caption{%
If $L$ lies in $\Omega_{1A}$
and $x(0)> X(s_1(0),s_2(0))$,
then the solution converges to the periodic orbit.
The parameters are
the values given in Example \ref{ex_N0},
and the initial condition is $(s_1(0),s_2(0),x(0))=(0.23,0.6,0.53)$.
}
\label{fig_cycling}
\end{figure}

In the following Corollary, we consider a case in which we are guaranteed that 
 $X(s_1^0,s_2^0)\le 0$.
In this case, by Proposition \ref{prop_Omega1B}, it follows that 
any solution of \eqref{modeleq}
with $(s_1(0),s_2(0))=\Omega_1$ and $x(0)>0$
converges to the periodic orbit.

If $\bar{s}_1>\lambda_1$ and $\bar{s}_2>\lambda_2$,
we define $\Omega_\lambda$ to be the region in the $s_1$-$s_2$
plane   that lies  between the
lines $s_2 = \bar{s}_2+R_{21}(s_1-\lambda_1)$ and
$s_2=\lambda_2+R_{21}(s_1-\bar{s}_1)$ and above or to the left of
$\Gamma^-$,
i.e.,
\begin{equation}
\Omega_\lambda =\left\{
    (s_1,s_2): \begin{aligned}
      &s_1>\bar{s}_1 \;\text{or}\; s_2>\bar{s}_2,\;\text{and}\\
      &
      V(\lambda_1,\bar{s}_2)\le V(s_1,s_2)\le V(\bar{s}_1,\lambda_2)
    \end{aligned}
  \right\}.
\end{equation}
For every $(s_1,s_2)\in\Omega_\lambda$, 
we have $\min\{f_1(s_1),f_2(s_2)\} > D$, and so growth of $x$ is always positive in this region.

\begin{corollary} Assume $(s_1^\IN,s_2^\IN)\in \Omega_1$ and $\mu(r)>0$.
If $\bar{s}_1>\lambda_1$
and $\widetilde{s}_2\ge\lambda_2$, then any solution
$(s_1(t),s_2(t),x(t))$ with $(s_1(0),s_2(0))\in\Omega_\lambda$
and $x(0)>0$ converges to the unique periodic orbit of \eqref{modeleq}.
\label{cor_lambda}
\end{corollary}

\begin{proof}
First note that, since $\widetilde{s}_2\ge \lambda_2$, the line through $(\bar{s}_1^+,\bar{s}_2^+)$ is in $\Omega_\lambda$, and $\Omega_\lambda\cup \Omega_{1A}$ is connected. For any $(s_1^0,s_2^0)\in\Omega_\lambda$, we have $I(s_1^0,s_2^0)>0$ and $(g\circ \pi^-)(s_1^0,s_2^0)\in\Omega_\lambda\cup\Omega_{1A}$.
By \eqref{eq_X}, $X(s_1^0,s_2^0)<0
< x(0)$, and by Theorem \ref{thm_X}, $(s_1(t),s_2(t),x(t))$ converges to the periodic orbit.
\end{proof}

\begin{figure}[htbp]
{\includegraphics[trim = 1cm 0cm 1cm 0cm, clip, width=.49\textwidth]{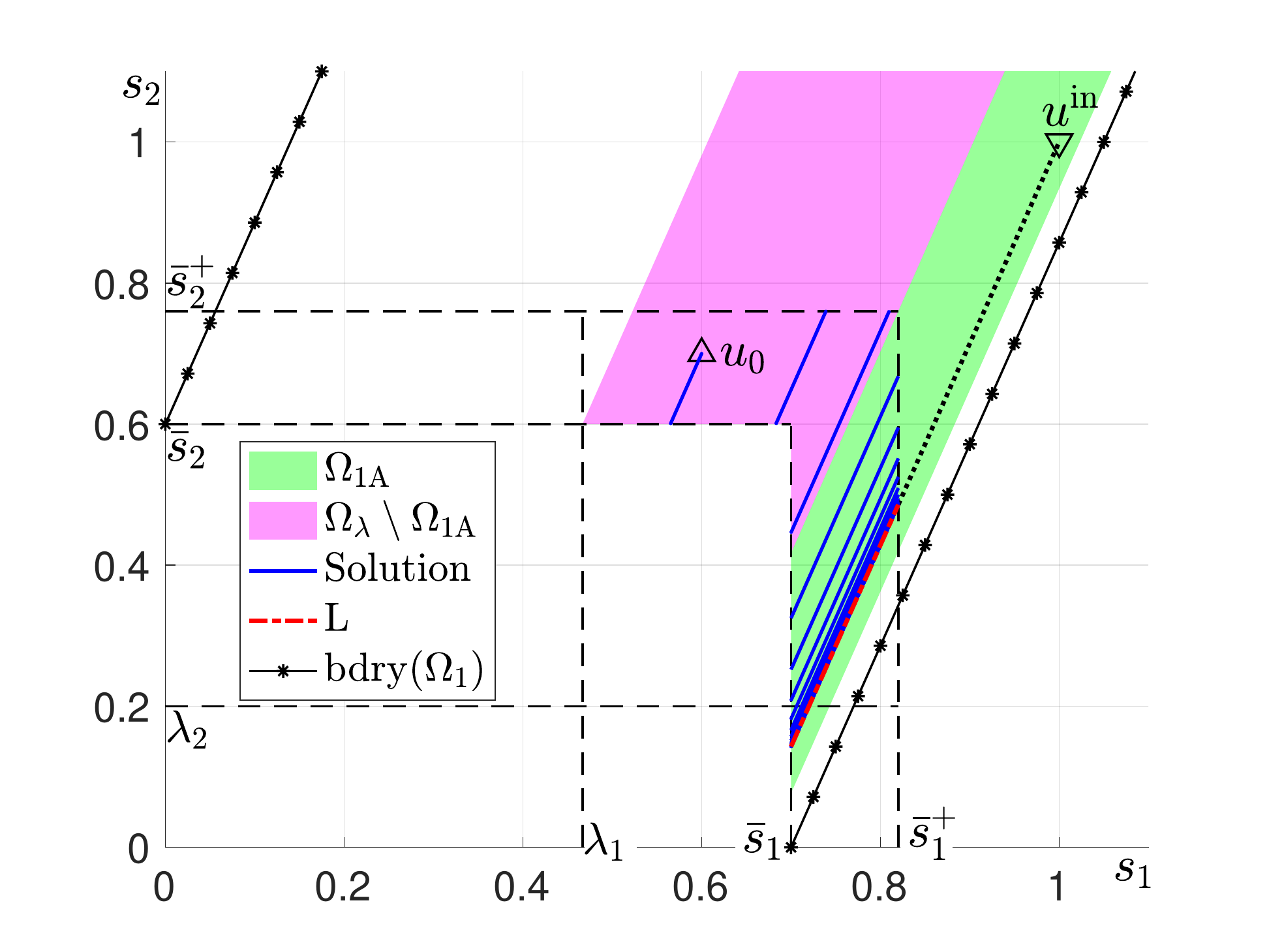}}
{\includegraphics[trim = 1cm 0cm 1cm 0cm, clip, width=.49\textwidth]{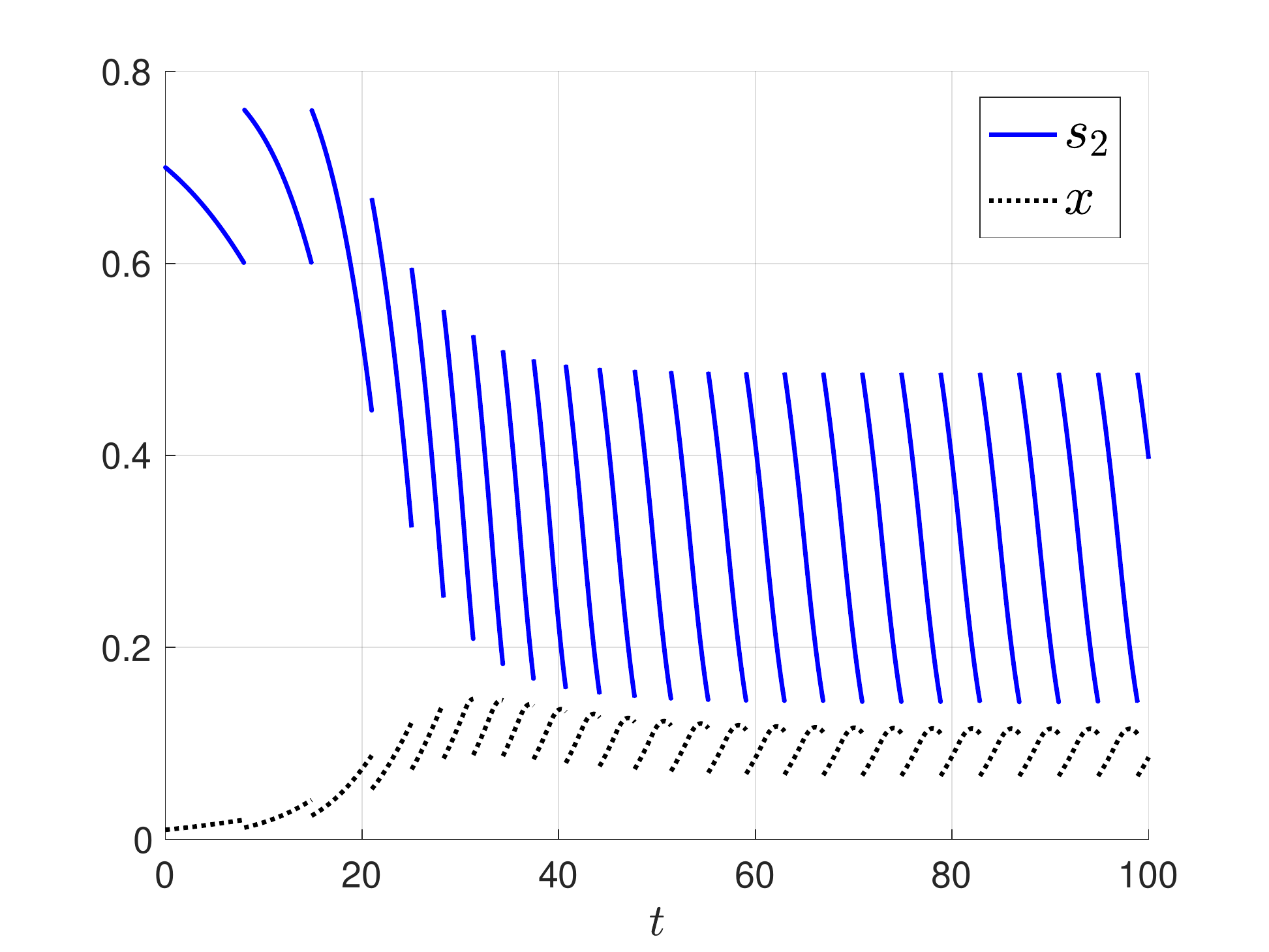}}
\caption{%
 If $u_0\in
\Omega_{\lambda} \setminus \Omega_{1A}$ and
$\Omega_{\lambda} \cup \Omega_{1A}$ is connected,
when $X(s_1^0,s_2^0)\le 0$,
the fermentation is still always successful for all $x(0)>0$.
The parameters are
the values given in Example \ref{ex_lambda}.
The initial condition is $(s_1(0),s_2(0),x(0))=(0.6,0.7,0.01)$.
}
\label{fig_lambda}
\end{figure}

\begin{example}
Consider \eqref{modeleq} with the Monod functional responses
$f_i(s_i)=\frac{m_is_i}{a_i+s_i}$, $i=1,2$,
and parameters 
$(m_1,m_2,a_1,a_2)=(2,2,1.4,0.6)$,
$(Y_1,Y_2,D)=(2,0.7,0.5)$,
and
$(\bar{s}_1,\bar{s}_2,s_1^\IN,s_2^\IN,r)=(0.7,0.6,1,1,0.4)$.
Then, \[
  (\bar{s}_1^+,\bar{s}_2^+)
  = (0.82,0.76),\;
  \widetilde{s}_2\approx 0.42,\;
  \widehat{s}_2\approx 0.14,\;
  \mu(r)\approx 0.04,\;
  V_-\approx -0.19.
\]
The equation $f_i(s_i)=D$, $i=1,2$, yields $\lambda_i=\frac{a_iD}{D-m_i}$,
which gives $\lambda_1\approx 0.4677$ and $\lambda_2=0.2$.
Since $\lambda_1<\bar{s}_1$ and $\lambda_2<\widetilde{s}_2$,
$\Omega_{\lambda}\cup \Omega_{1A}$
is a connected set,
and the hypotheses in Corollary \ref{cor_lambda} are satisfied.

We take the initial value $(s_1^0,s_2^0)= (0.6,0.7)$.
Then, $V(s_1^0,s_2^0)=-0.59$.
A direct calculation gives $V(\lambda_1,\bar{s}_2)\approx -0.79$,
so that $V(\lambda_1,\bar{s}_2)< V(s_1^0,s_2^0)<0$.
This implies that $(s_1^0,s_2^0)\in \Omega_\lambda$.
By Corollary \ref{cor_lambda},
the fermentation succeeds for every initial value $x(0)>0$.
An illustration is shown in Figure~\ref{fig_lambda}.
\label{ex_lambda}
\end{example}

\begin{remark} If $\lambda_2>\widetilde{s}_2$ then the set $\Omega_{1A}
\cup\Omega_\lambda$ is not connected.
Then for some $(s_1^0,s_2^0)\in \Omega_\lambda$, $(g\circ
\pi^-)(s_1^0,s_2^0)$ is in the gap between $\Omega_\lambda$ and
$\Omega_{1A}$. We are unable to rule out the possibility that
the net growth in this gap is negative, and so it is
conceiveable that $I((g\circ \pi^-)(s_1^0,s_2^0)) <0$. In
particular, we can choose $(s_1^0,s_2^0)\in \Omega_\lambda$
such that $I(s_1^0,s_2^0) < - I((g\circ \pi^-)(s_1^0,s_2^0))$.
Therefore, 
\begin{equation}
X(s_1^0,s_2^0) \geq -I(s_1^0,s_2^0) - (1-r)^{-1}
I((g\circ \pi^-)(s_1^0,s_2^0))>0.
\end{equation}
  Therefore, for some positive initial concentrations of biomass, the reactor will fail, even though the initial conditions are in $\Omega_\lambda$.
\end{remark}

\section{Maximizing the Output}
\label{sec_max}
In this section we regard $r$ as a variable in the interval $(r_*,1)$,
where $r_*$ is the number given in Theorem \ref{thm_per_exists}.

For each $r\in (r_*,1)$,
there is a periodic orbit.
In each period, there is exactly one impulse.
As shown in the proof of Theorem \ref{thm_per_exists},
the left and right limits at an impulse
are, respectively, \[
  (\bar{s}_1,\widehat{s}_2,x_-(r))
  \quad\text{and}\quad
  (\bar{s}_1^+(r),\widehat{s}_2^+(r),x_+(r)),
\] where \begin{equation}\label{limit_x}
  x_-(r)
  =\frac{1}{r}\mu(r)
  \quad\text{and}\quad
  x_+(r)
  =\frac{1-r}{r}\mu(r).
\end{equation}

The trajectory of the periodic orbit
can be parametrized by \[
  s_2
  = \widehat{s}_2+R_{21}(s_1-\bar{s}_1)
\]
and \begin{align}
  x=X(s_1; r)
  &= x_-(r)
  - Y_1\int_{\bar{s}_1}^{s_1}
  1-\frac{D}{\min\big\{f_1(v),f_2(\widehat{s}_2+R_{21}(v-\bar{s}_1))\big\}}
  \;dv
  \label{eq_xs1m}
  \\
  &= x_+(r)
  + Y_1\int_{s_1}^{\bar{s}_1^+}
  1-\frac{D}{\min\big\{f_1(v),f_2(\widehat{s}_2+R_{21}(v-\bar{s}_1))\big\}}
  \;dv
  \label{eq_xs1p}
\end{align}
with $s_1\in (\bar{s}_1,\bar{s}_1^+)$
and $\bar{s}_1^+= \bar{s}_1+ r(s_1^{\mathrm{in}}-\bar{s}_1)$.

Denote the minimal period of the periodic orbit by $T(r)$.
Then \begin{equation}\label{eq_T}
  T(r)
  = Y_1\int_{\bar{s}_1}^{\bar{s}_1^+(r)}
    \frac{1}{\min\big\{f_1(v),f_2(\widehat{s}_2+R_{21}(v-\bar{s}_1))\big\}\; X(v;r)}
  dv.
\end{equation}
In the long run, the average amount of  output
divided by the total volume
is \begin{equation*}
  Q(r)= \frac{r}{T(r)}.
\end{equation*} 
Maximizing $Q(r)$ for $r\in (r_*,1)$ is equivalent to maximizing the   output.

\begin{lemma}
\label{lem_limitT}
The minimal period $T(r)$, $r\in (r_*,1)$ of the periodic orbit of \eqref{modeleq}
satisfies $\lim_{r\to 1^-}T(r)=\infty$.
Also $\lim_{r\to r_*}T(r)=\infty$ if $r_*>0$.
\end{lemma}

\begin{proof}
As $r\to 1$,
we have $x_+(r)\to 0$.
By \eqref{eq_T} and \eqref{eq_xs1p},
$T(r)\to \infty$.

If $r_*>0$, then by \eqref{limit_x}, $x_-(r)\to 0$ as $r\to r_*$.
Along the periodic orbit,
$x\ge x_-(r)>0$
and \[
  \min\{f_1(s_1),f_2(s_2)\}\ge \min\{f_1(\bar{s}_1),f_2(\bar{s}_2)\}>0.
\]
By \eqref{eq_T} we conclude that $T(r)\to \infty$ as $r\to r_*$.
\end{proof}

\begin{proposition}
\label{prop_Q0}
The function $Q(r)=r/T(r)$, $r\in (r_*,1)$,
satisfies $\lim_{r\to 1}Q(r)=0$.
If $r_*>0$,
then $\lim_{r\to r_*}Q(r)=0$.
If $r_*=0$,  $\bar{s}_1\ge\lambda_1$, and $\widehat{s}_2\ge\lambda_2$,
then $\lim_{r\to r_*}Q(r)=\min\{f_1(\bar{s}_1),f_2(\widehat{s}_2)\} -D\ge 0$.
\end{proposition}

\begin{proof}
By Lemma \ref{lem_limitT}, $\lim_{r\to 1}Q(r)= 1/ (\lim_{r\to 1}T(r))=0$.

If $r_*>0$, then,
by Lemma \ref{lem_limitT},
$\lim_{r\to r_*}Q(r)= r_*/ (\lim_{r\to r_*}T(r))=0$.

Next we assume $r_*=0$.
By \eqref{eq_T} \begin{equation}
  Q(r)
  =r\left/
  \left(Y_1\int_{\bar{s}_1}^{\bar{s}_1^+(r)}
    \frac{1}{\min\big\{f_1(v),f_2(\widehat{s}_2+R_{21}(v-\bar{s}_1))\big\}\; X(v;r)}
  \;dv\right)
  \right.,
  \label{Qint}
\end{equation}
where $X$ is defined by \eqref{eq_xs1m}.
Since $X(\bar{s}_1;r)=x_-(r)=\frac{\mu(r)}{r}$,
using L'H\^{o}pital's rule
and the definition of $\mu(r)$ in \eqref{def_mu},
\[
  \lim_{r\to 0}X(\bar{s}_1;r)
  = Y_1\left.\frac{d\bar{s}_1^+(r)}{dr}\right|_{r=0}
  \left(1-\frac{D}{\min\{f_1(\bar{s}_1),f_2(\widehat{s}_2)\}}\right).
\]
Since $\bar{s}_1^+= \bar{s}_1+ r(s_1^{\mathrm{in}}-\bar{s}_1)$,
we have \begin{equation}
  \left.\frac{d\bar{s}_1^+(r)}{dr}\right|_{r=0}
  = s_1^\IN-\bar{s}_1,
  \label{lims1r0}
\end{equation}
and it follows that
\begin{equation}
  \lim_{r\to 0}X(\bar{s}_1;r)= Y_1(s_1^\IN-\bar{s}_1)\left(1-\frac{D}{\min\{f_1(\bar{s}_1),f_2(\widehat{s}_2)\}}\right).
  \label{limXr0}
\end{equation}
Furthermore, we have 
\begin{align}
T'(r)
=&  \frac{Y_1(s_1^\IN-\bar{s}_1)}{\min\big\{f_1(\bar{s}_1^+),f_2(\widehat{s}_2+rR_{21}(s_1^\IN-\bar{s}_1))\big\}\; X(\bar{s}_1^+;r)}\notag\\
&-\frac{Y_1}{r^2}\int_{\bar{s}_1}^{\bar{s}_1^+(r)} \frac{\mu'(r)r-\mu(r)}{\min\big\{f_1(v),f_2(\widehat{s}_2+R_{21}(v-\bar{s}_1))\big\}\; X(v;r)^2}dv.\label{eq:Tprime}
\end{align}
Since $\bar{s}_1^+(0)=\bar{s}$
and $X(\bar{s}_1;0)\ne 0$,
by \eqref{limXr0} and \eqref{eq:Tprime}
we obtain \[
  T'(0)
  =\frac{Y_1(s_1^\IN-\bar{s}_1)}{\min\big\{f_1(\bar{s}_1),f_2(\widehat{s}_2)\big\}\; X(\bar{s}_1;0)}
  = \frac{1}{\min\{f_1(\bar{s}_1),f_2(\widehat{s}_2)\} -D}.
\]
From the expression $Q(r)=r/T(r)$,
using L'H\^{o}pital's rule, we conclude that
$\lim_{r\to 0}Q(r)=1/T'(0)=\min\{f_1(\bar{s}_1),f_2(\widehat{s}_2)\} -D$.
\end{proof}

 Assume $r_*>0$. Since $\lim_{r\to r_*}Q(r)=\lim_{r\to 1}Q(r)=0$,
by Proposition \ref{prop_Q0},
$Q(r)$ attains its maximum at some value of $r$ in $(r_*,1)$.
Unfortunately,
the analytical expression of the derivative $\frac{d}{dr}Q(r)$
is too complicated for finding a critical point of $Q(r)$.
We have only obtained the maximum value using a numerical simulation.
An illustration of the maximal value of $Q(r)$ is given in Figure~\ref{fig_RTQ}.

\begin{figure}[htbp]
\centering
\begin{tabular}{cc}
\includegraphics[trim = .5cm 0cm 1cm 0cm, clip, width=.48\textwidth]{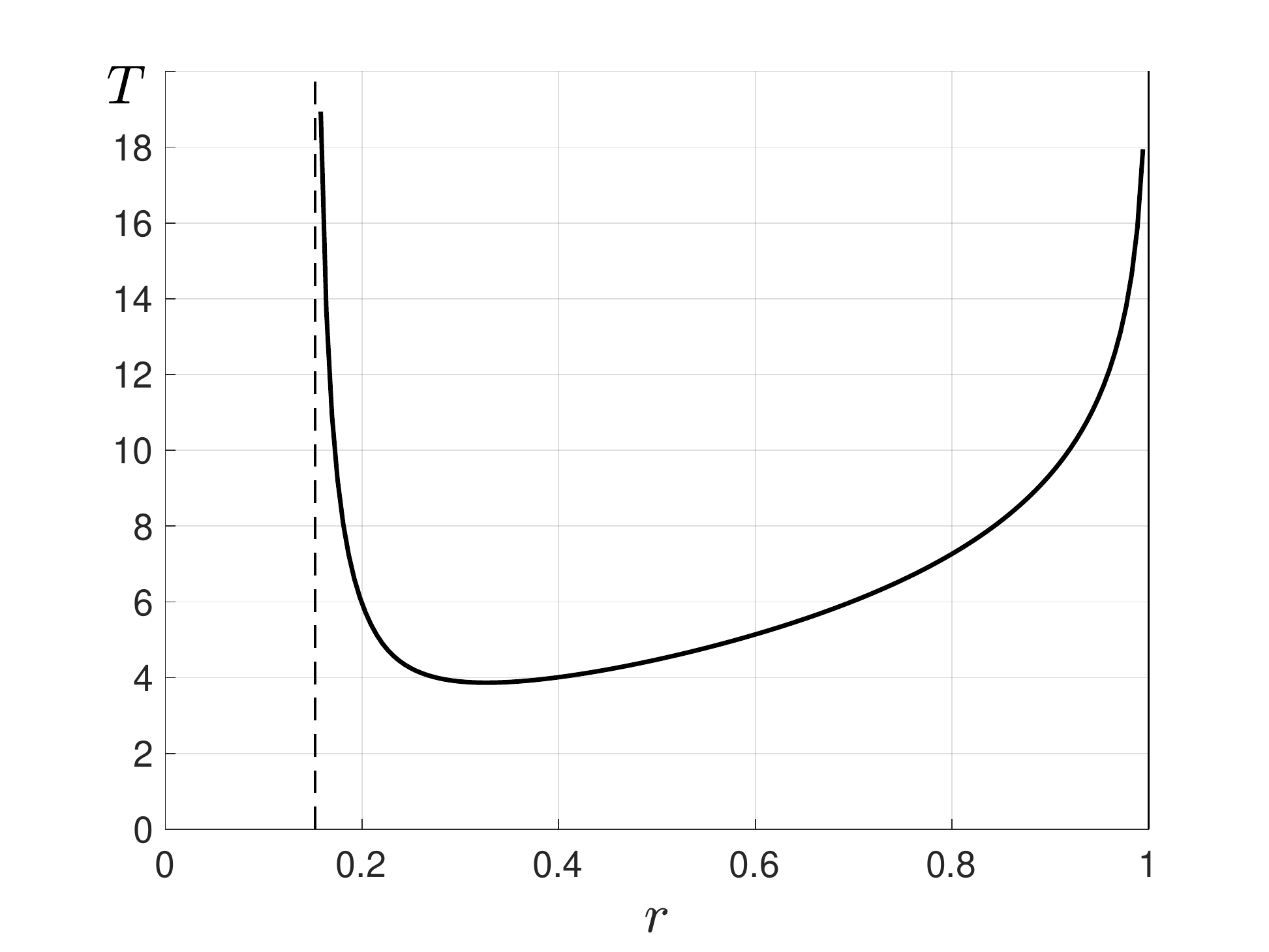}
&
\includegraphics[trim = .5cm 0cm 1cm 0cm, clip,width=.48\textwidth]{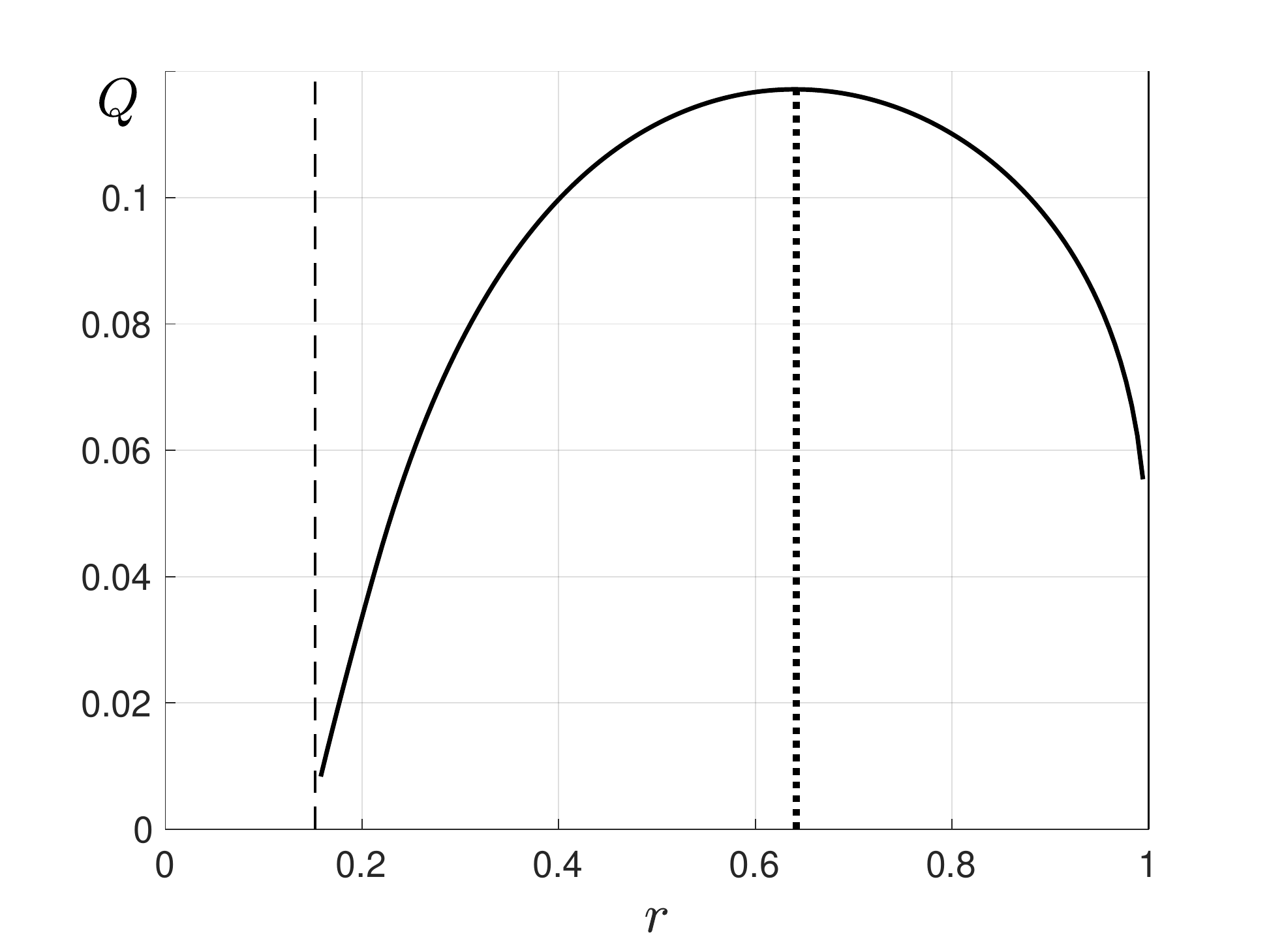}
\\
(A)&(B)
\end{tabular}
\caption{%
(A)
$T(r)$ is the minimal period of the periodic orbit, $r\in (r_*,1)$. 
The dashed line is $r=r_*$.
As $r\to r_*$ or $r\to 1$, $T(r)\to \infty$.
(B) The maximum of $Q(r)$ is attained at $r\approx 0.6416$,
indicated by the dotted line.
In the simulation, 
the response functions are
$f_i(s_i)=\frac{m_is_i}{a_i+s_i}$, $i=1,2$,
with parameters given in Example \ref{ex_lambda}.
}
\label{fig_RTQ}
\end{figure}

\section{Discussion}
\label{sec_discussion}
We have modeled the self-cycling fermentation process
assuming that there are two essential resources $s_1$ and $s_2$ that
are growth limiting for a population of microorganisms, $x$, using a system of impulsive differential equations with state-dependent impulses. Assuming that the process is  used for an
application such as water purification, where the resources $s_1$ and
$s_2$ are the pollutants, we assume that the threshold
for emptying and refilling a fraction of the contents of the
fermentor, resulting in the release of  treated water,  occurs when the concentrations
of both pollutants reach an acceptable concentration set by some
governmental agency.    We called these thresholds, 
$s_1\leq\bar{s}_1$ and $s_2\le \bar{s}_2$. We consider the process successful if once
initiated, it proceeds indefinitely without a need for any subsequent
interventions by the operator.

By solving the associated ODE system for $s_2$ in terms of $s_1$, we
show that solutions,  when projected onto the $s_1$-$s_2$ plane, are lines
with  slope given by  the ratio of the growth yield constants.
In order to derive  necessary conditions for successful operation of the
fermentor, we first divide the $s_1$-$s_2$ plane
into two regions: $\Omega_0$ and $\Omega_1$. The  model predicts that solutions of the
associated system of ODEs with initial conditions in  $\Omega_0$
approach the axes without ever reaching the
thresholds for emptying and refilling and the  reactor  fails,
independent of the initial concentration of microorganisms.
Solutions of the associated system of ODEs with initial
conditions in $\Omega_1$ have the potential to reach the  threshold
for emptying and refilling, but in this case, successful operation
can also depend on the initial concentration of the
population of microorganisms.

In most cases, at startup the
input concentration of the pollutant
would be the concentration of the pollutant in the
environment, which we are assuming is constant, i.e.,
$(s_1(0),s_2(0))=(s_1^\IN,s_2^\IN)$.  If, for any solution starting at
these input
concentrations of the resources, $(s_1^\IN,s_2^\IN)$, and positive
concentration of biomass, $x(0)>0$,   the
 threshold for emptying and refilling, $s_1\leq\bar{s}_1$ and $s_2\le \bar{s}_2$, is
reached with net positive growth of the biomass, the model analysis predicts that  we can choose an
emptying/refilling fraction, $r$,  so that the system cycles indefinitely.  In
this case 
the solution   approaches a  periodic solution with one impulse per period.

If the system has
a periodic solution, the $(s_1,s_2)$ components of the periodic orbit lie
along the line with slope given by  the ratio of the growth yield
constants joining $(s_1^\IN,s_2^\IN)$  and the point in the $s_1$-$s_2$
plane where both thresholds are reached. 
The net change in the biomass on the periodic orbit, that we denote
$\mu(r)$, must then also be positive. 

For other initial conditions in  $\Omega_1$, in order
for the process to operate successfully,   it is not enough that
$\mu(r)>0$.  There
is also a minimum concentration of biomass, $X$, that depends on the initial
concentration of the resources, that is  required for the
reactor to be successful.  If the initial concentration of biomass is
larger than $X$,  then our analysis predicts that 
the reactor will cycle indefinitely and solutions will converge to the
periodic orbit. If the initial concentration of biomass is less than or
equal to $X$, then the reactor will cycle a finite number of times and then fail. If there is no periodic orbit, then the reactor will either
cycle a finite number of times and then fail, or will cycle
indefinitely, but   the time between cycles will approach infinity.

Besides depending on the initial concentration of the resources at start
up, the minimum concentration of biomass at startup  required for
successful operation depends on the
emptying/refilling fraction in an interesting way. The closer $r$ is  to
one, the  smaller the number of impulses that are required for solutions to get
to the periodic orbit. However, the time spent in a region of negative
growth could be larger, and so $X$ would be larger. The closer $r$ is  to
zero, results in  less time  spent in regions with negative growth, but
more impulses are then required to get close to the periodic orbit. Each
impulse removes biomass from the reactor, and so $X$ would also increase.
This implies that there is an optimal value of $r$ for which the reactor
has the best potential for success.
The values of the growth yield constants, $Y_1$ and
$Y_2$, also play a role in the size of $X$. If their ratio  is held
constant, but each value is scaled by a constant $c>0$, then $X$ is also
scaled by the same constant $c$.  Knowing this is important when
selecting the population of microorganisms to use in the process.

If the choice of potential 
microorganisms for use in the process is restricted, then it might be
easier to treat more highly polluted water than less polluted water,
provided the microorganisms are not inhibited at high concentrations of the pollutant. We have shown that for
successful operation, it is necessary that an $r$ exists such that $\mu(r)>0$.  
One way to increase $\mu(r)$ without changing anything else is to increase $s_1^\IN$ and $s_2^\IN$ in
such a way that it still lies on the same line as before.
It is also important to choose a population of microorganisms so that
$(s_1^\IN$,$s_2^\IN)$  lies in $\Omega_{1}$.  It might only be possible
to do this by increasing the concentration of one of the pollutants.
However, another possibility might be to pre-process the input with a
different population of microorganisms that 
moves  $(s_1^\IN$,$s_2^\IN)$ into an 
acceptable position so that a second population can then treat the  water
effectively.

We also make what might appear to be  other surprising
observations.
Although the break-even concentrations play a role,  it is not necessary for both break-even
concentrations to be below their respective thresholds for emptying
and refilling  for the process to be successful  (see  Figure~\ref{fig_cycling}). Also, the process can still fail
when
both break-even concentrations are below
their respective thresholds 
(see Figure~\ref{fig_washout}). 

For growth on a single, non-inhibitory, limiting resource in the self-cycling fermentation process, it has been shown that when the system has a periodic orbit, every solution either converges to the periodic orbit, or converges to an equilibrium without a single impulse \cite{Smith2001}. If the resource is inhibitory at high concentrations, it has been shown that solutions may also converge to an equilibrium after a single impulse, but if there are at least two impulses then the solution is destined to converge to the periodic orbit\cite{Fan2007}. In contrast, if there are two limiting essential resources, we have shown that there may be many impulses before the system converges to an equilibrium, even when the system has a periodic orbit.  The example in Figure \ref{fig_failed} demonstrates failure after two impulses. 

An important issue when setting up the self cycling fermentation
process is   the choice of the emptying and refilling fraction, $r$.
In the application we considered we were interested in optimizing the
total amount of output.  In the example,
shown in section \ref{sec_max}, we demonstrated that the optimal value of the emptying/refilling fraction is $r\approx 0.64$. This result is consistent with what was shown in the single resource cases \cite{Fan2007,Smith2001}.
Another reason for implementing a self-cycling fermentation process
instead of a continuous input process is to maximize the concentration of some
microorganism in the output over some time period. For example, one
recent `proof of concept' study \cite{Wang2017}   investigated using the self-cycling fermentation process to improve the production of cellulosic ethanol production. In their investigation, and many other applications of self-cycling fermentation the emptying/refilling fraction $r$ is set to one half. While this 
is convenient for experiments and measurements, our results indicate that
this is might not be  the  optimal choice of $r$. 

\titleformat{name=\section}[runin]{}{\thetitle.}{0.5em}{\bfseries}
\def\refname{\centerline{\small REFERENCES}\\[-.5em]}

\end{document}